\documentclass[12pt,a4paper]{article}
\usepackage{preamble}

\title{Special lines on contact manifolds}
\author{\JaBu \and Grzegorz Kapustka \and Micha\l{} Kapustka}

\hypersetup{%
  pdftitle={Special lines on contact manifolds},%
  pdfauthor={Jaroslaw Buczynski,  Grzegorz Kapustka,  Michal Kapustka}}

\date{3rd June 2020}

\usepackage{color}
\newenvironment{red}{\color{red}}{}
\newcommand{\bred}{\begin{red}}
\newcommand{\ered}{\end{red}}

\newenvironment{blue}{\color{blue}}{}
\newcommand{\bblue}{\begin{blue}}
\newcommand{\eblue}{\end{blue}}

\newenvironment{green}{\color{green}}{}
\newcommand{\bgreen}{\begin{green}}
\newcommand{\egreen}{\end{green}}

\usepackage[textsize=tiny]{todonotes}

\theoremstyle{plain}

\begin{document}

\maketitle

\begin{abstract}
In a series of two articles Kebekus studied deformation theory of minimal rational curves on contact Fano manifolds.
   Such curves are called contact lines.
   Kebekus proved that a contact line through a general point is necessarily smooth
      and has a fixed standard splitting type of the restricted tangent bundle.
   In this paper we study singular contact lines and those with special splitting type.
   We provide restrictions on the families of such lines,
   and on contact Fano manifolds which have reducible varieties of minimal rational tangents.
   We also show that the results about singular lines naturally generalise to complex contact manifolds,
      which are not necessarily Fano, 
      for instance,
      quasi-projective contact manifolds or compact contact manifolds of Fujiki class~$\ccC$.
   In particular, in many cases the dimension of a family of singular lines is at most $2$ less 
      than the dimension of the contact manifold.
\end{abstract}

\medskip
{\footnotesize
\noindent\textbf{addresses:\\}
J.~Buczy\'nski, \nolinkurl{jabu@mimuw.edu.pl}, 
   Institute of Mathematics of the Polish Academy of Sciences, ul. \'Sniadeckich 8, 00-656 Warsaw, Poland,
   and Faculty of Mathematics, Computer Science and Mechanics, University of Warsaw, ul.~Banacha 2, 02-097 Warszawa, Poland\\
G.~Kapustka, \nolinkurl{grzegorz.kapustka@uj.edu.pl},
Faculty of Mathematics and Computer Science of the Jagiellonian University,
\L{}ojasiewicza 6, 30-348 Krak\'ow, Poland\\
M.~Kapustka, \nolinkurl{michal.kapustka@impan.pl}, 
Institute of Mathematics of the Polish Academy of Sciences, ul. \'Sniadeckich 8, 00-656 Warsaw, Poland\\

\noindent\textbf{keywords:}
complex contact manifold, minimal rational curves, contact lines, VMRT, manifolds of Fujiki class C.

}

\newpage

\tableofcontents

\section{Introduction}\label{sect_introduction}

A contact manifold is a complex manifold of odd dimension $2n+1$ with a contact structure, that is, 
an exact sequence  
\begin{equation}\label{equ_contact_sequence_intro}
  0 \to F \to TX \stackrel{\theta}{\to} L \to 0,
\end{equation}
where $F$ is a subbundle of $TX$, $L$ is a line bundle,
and such that the locally defined derivative of $\theta$ induces a symplectic structure on each fibre of $F$.
This article addresses the problem of classification of contact manifolds. 
A lot of work has already been done in this direction, see \cite{jabu_contact_duality_and_integrability}, \cite{jabu_moreno_contact_survey}
  for an overview and motivation in the projective case.
The major remaining part of the task, is to classify contact Fano manifolds, which are expected to be always homogeneous spaces.
In addition non-projective contact manifolds recently gained attention, 
  see \cite{hwang_manivel_quasi_complete_contact}, \cite{frantzen_peternell}, \cite{peternell_jabu_contact_3_folds}, 
  \cite{peternell_schrack_contact_Kaehler_mflds}, and the classification in the non-projective case is widely open.
Even in dimension three, the classification of rationally connected compact contact manifolds with $b_2 \ge 2$ is unknown.
Although projective contact manifolds are our main object of study, in this paper we also provide some extensions of our results to the settings involving generically contact manifolds or non-projective contact manifolds.

Among the main tools to approach the classification problem is the theory of minimal rational curves. This article is among the first attempts to study minimal rational curves on a projective manifold $X$
  without assuming that they are general, or that they pass through a general point.
In the case of contact Fano manifold $X$ of dimension $2n+1$ it amounts to study \emph{contact lines}, that is, rational curves for which the pullback of the restriction of the canonical bundle on $X$ through the normalisation is 
$\mathcal{O}_{\PP^1}(-n-1)$.
Equivalently, the degree of the pullback of the line bundle $L$ as in \eqref{equ_contact_sequence_intro} is $1$.
It is well known that unless $X \simeq \PP^{2n+1}$, 
  those contact lines exist, cover $X$ and form a family of pure dimension $3n-1$. Moreover,  for a general contact line with a parametrisation $f \colon \PP^1 \to X$
  the pullback of the tangent bundle $TX$ has a standard splitting type:
\begin{equation} \label{equ_standard_splitting_type}
  f^*TX \simeq \ccO^{\oplus (n+1)} \oplus \ccO(1)^{\oplus (n-1)} \oplus \ccO(2) = \ccO(0^{n+1}, 1^{n-1}, 2).
\end{equation}
Contact lines satisfying \eqref{equ_standard_splitting_type} are called \emph{standard}. Unfortunately, it is not known if all contact lines on a contact manifold are standard. This is the case for homogeneous manifolds, which are expected to be all Fano contact manifolds. 
In general, by results of Kebekus \cite{kebekus_lines1}
   any contact line through a general point is smooth and standard. 

Moreover, \cite[Thm~3.1]{kebekus_lines2} claimed that for general $x\in X$ the variety $\ccH_x$ parameterising  contact lines passing through $x$ is irreducible.
Unfortunately, there is a gap in the proof, 
   see \cite[Rem.~3.2]{jabu_contact_duality_and_integrability} or 
   \S\ref{sec_history_lines} for more details.
The argument of Kebekus only shows that if for a general $x\in X$ the variety $\ccH_x$ is reducible, then the set of non-standard lines forms a divisor in a component of all the lines.
In this article, we proceed to further study the properties of the potential contact Fano manifold $X$, which contains many non-standard lines. More precisely, we provide a description of the locus in $X$ swept out by the non-standard lines.

\begin{thm}\label{thm_dimension_of_the_space_of_special_lines}
  Let $X$ with $F$ and $L$ as in \eqref{equ_contact_sequence_intro} be a contact Fano manifold. Let $\ccH$ be an irreducible component of the Hilbert scheme parametrising contact lines on $X$. 
      Suppose that $\ccB \subset \ccH$ is an irreducible codimension 1 subset,
      consisting only of non-standard lines.
   Let $B \subset X$ be the locus swept by those lines. Then:
   \begin{enumerate}
    \item \label{item_B_nonnormal}
         $B$ is a non-normal irreducible divisor on $X$. 
    \item \label{item_normalisation_of_B_is_a_Pn_bundle}
         Denote by $\xi\colon \ccU\to B$ the normalisation map.
         There exists a normal variety $\ccR$, a vector bundle $E$ of rank $n+1$ on $\ccR$,
            and an isomorphism $\ccU \simeq \PP (E)$ such that 
            $E$ is isomorphic to $(\pi_*\xi^* L)^*$, where  $\pi \colon \ccU \to\ccR$ is the canonical projection.
    \item \label{item_normalisation_of_B_normalises_linear_spaces}
          The restriction $\xi|_{\PP^n}$ to any fibre of $\pi\colon\ccU \to \ccR$ is 
             the normalisation of the image $\xi(\PP^n)$ and $(\xi|_{\PP^n})^* L = \ccO_{\PP^n}(1)$.
   \end{enumerate}
\end{thm}
The theorem is proved in Section~\ref{sect_divisors_of_non_standard_lines}. 
The proof of part \ref{item_B_nonnormal} can be extracted from  the arguments in \cite[\S3]{kebekus_lines2}.
See also Lemma~\ref{lem_ccB_x_is_a_union_of_irreducible_components}, which shows $\dim B =2n$,
   and \S\ref{sec_history_lines}, which reviews the content of and the gap in \cite[Prop.~3.2]{kebekus_lines2}. 
An alternative way to see that $B$ is not normal is using \ref{item_normalisation_of_B_is_a_Pn_bundle},
   see the last paragraph of the proof of the theorem.

Our further results concern singular and non-standard lines on generalisations of projective contact manifolds.

Let $X$ be a complex manifold and 
suppose that there exist a vector bundle $F$ 
  and a line bundle $L$ admitting an exact sequence 
  as in \eqref{equ_contact_sequence_intro}, 
  but that the derivative of the locally defined 
  $1$-form $\theta\colon TX \to L$ 
  determines a symplectic form  only on an open dense subset of $X$.
In such situation we will say that $(X,F)$ is a 
  \emph{generically contact manifold}.

In this setting we can measure the degree of rational curves using the intersection with $L$.
In particular, the \emph{contact lines} are the (complete) rational curves $C \subset X$ with the intersection number $L.C =1$.
The notion of standard line extends to this context. However, contrary to the Fano case, there is no guarantee that lines exist, and the intersection of $L$ with a rational curve may potentially be zero or negative. 
This is a major issue, however, assuming that lines exist, we obtain many results that are parallel to the contact projective case.

Recall that a complex manifold is of \emph{Fujiki class $\ccC$}, 
   if it is bimeromorphic to a compact K\"ahler manifold
  \cite{fujiki_closedness_of_Douady_spaces}. 

\begin{thm}\label{thm_singular_lines_on_non_projective}
  Suppose that $X$ is either quasi-projective or compact of Fujiki class $C$ 
  and in addition $(X,F)$ is a generically contact manifold. 
  We have the following restrictions on the families of singular contact lines.
  \begin{enumerate}
   \item \label{item_singular_lines_do_not_cover_generically_contact}
         Singular contact lines do not cover $X$.
   \item
   \label{item_singular_lines_have_small_dim_on_projective_generically_contact}
         If $X$ is projective and $L$ is ample, 
           then the dimension of the space parametrising the singular lines 
           is at most $\dim X -2$. 
   \item   
   \label{item_singular_lines_sweeping_codim_1_subset_of_non_projective_contact}
         If $(X,F)$ is contact and a family of singular lines sweeps out 
           a locus of codimension $1$ in $X$,
           then the dimension of this family is equal to $\dim X -2$.
  \end{enumerate}
\end{thm}
This theorem generalises \cite[Prop.~3.3]{kebekus_lines1} which is a similar statement 
   for a contact Fano manifold $(X,F)$.
We prove the theorem in Subsection~\ref{sect_sing_lines_and_distributions}, 
  Corollaries~\ref{cor_singular_lines_on_generically_contact} and \ref{cor_singular_lines_on_non_projective_contact}.

\begin{prop}\label{prop_line_through_general_point_is_standard_on_non_projective}
   Suppose that $(X,F)$ is a contact manifold.
   If $X$ is either quasi-projective or compact of Fujiki class $C$, 
   then any line through a general point $x \in X$ is standard.
\end{prop}
This  result generalises \cite[Lemma~3.5]{kebekus_lines1}, 
   which again is for contact Fano $(X,F)$. 
We show the proposition in Subsection~\ref{sect_splitting_types}, 
  Corollary~\ref{cor_non_standard_lines_do_not_cover_X}.

The structure of the paper is the following. 
In Section~\ref{sect_context}, we outline the content of relevant literature and sketch our arguments.
In Section~\ref{sect_preliminaries}, we discuss and review parameter spaces for subvarieties and cycles, namely Barlet spaces and Chow varieties. 
In a setting of polarised complex manifolds 
we define lines, linear spaces and their families and relate these families to cycle spaces.
We describe the situation when there are plenty of lines on a variety.
Section~\ref{sect_distributions}, describes the notions of a distribution and a manifold with a global corank $1$ distribution.
In Section~\ref{sect_prelim_contact_manifolds}, we review the literature on contact manifolds and discuss possible splittings of tangent bundle restricted to lines. 
In Section~\ref{sect_sing_lines}, we investigate singular lines on polarised manifolds.
In Section~\ref{sect_divisors_of_non_standard_lines}, we show that if there are many non-standard lines on a contact Fano manifold, then their configuration must be very special.

\subsection*{Acknowledgements}

The authors are very grateful to Weronika Buczy\'nska, Stephane Druel, Daniel Greb, Andreas H\"oring Jun-Muk Hwang, Stefan Kebekus, Mateusz Micha{\l}ek,
   Thomas Peternell, {\L}ukasz Sienkiewicz, Luis Sola Conde, and Jaros{\l}aw Wi\'s\-niew\-ski for all their comments and discussions, 
   and to Daniel Barlet for his hints about cycle spaces.
We also thank Louise Hibberd for her suggestions improving the presentation of the article.
The authors are
   supported
   by the research project
  ``Uk\l{}ady linii na zespolonych rozmaito{\'s}ciach kontaktowych oraz uog{\'o}lnienia''
   funded by the Polish government science programme in 2012-2014.
In addition, J.~Buczy\'nski is supported by the scholarship ``START'' of the Foundation for Polish Science.

A thorough rewrite of the article followed suggestions by anonymous referees, whom we sincerely thank.
The revisions were made while Buczy\'nski was supported by the research  grant  from Polish National Science Centre, number 2013/08/A/ST1/00804,
   and by a scholarship of Polish Ministry of Science.
Moreover, these revisions took place partially during the Polish Algebraic Geometry mini-Semester (miniPAGES), 
   which was supported by the grant 346300 for IMPAN from the Simons Foundation and the matching 2015-2019 Polish MNiSW fund.
We thank IMPAN and the participants of the minisemester for an excellent and inspiring atmosphere during the event.

\section{Context and history}\label{sect_context}

In this section we present some of the relevant facts from the literature about contact manifolds. 
In several cases, if the references are not explicit, but the statements we present are well known to the experts, we fill in the arguments. 
Typically, these brief proofs simply combine a couple of references.
We also set some notation which is relevant for the main parts of the article.
The most important is the following setting.

\begin{setting}[contact projective setting]\label{sett_contact_projective}
     Suppose that $X$ is a complex projective manifold of dimension $2n+1$ with a contact distribution $F \subset TX$. 
     Denote $L=T X /F$.
\end{setting}

\subsection{Historical remarks on classification of contact manifolds}\label{sect_historical_remarks}

We recall the major classification result:

\begin{thm}[{\cite{4authors, demailly}}]\label{thm_classification}
  In Setting~\ref{sett_contact_projective},
         either $X= \PP(T^* M)$ for some projective manifold $M$ and $L\simeq \ccO_{\PP(T^* M)}(1)$,
         or $X$ is Fano and $\Pic X = \ZZ[L]$,
         or $X = \PP^{2n+1}$ and $L=\ccO_{\PP^{2n+1}}(2)$.
\end{thm}

If we suppose in addition to the assumptions of Theorem~\ref{thm_classification} that $X$ is Fano, then the case of $\PP(T^*M)$ becomes more specific, 
   and we have the uniqueness of the contact structure.
\begin{prop}\label{prop_classification_Fano}
  In Setting~\ref{sett_contact_projective}, assume in addition that $X$ is Fano.
  Then either $X= \PP(T^* \PP^{n+1})$ (and $L\simeq \ccO_{\PP(T^* \PP^{n+1})}(1)$),
         or $\Pic X = \ZZ[L]$,
         or $X = \PP^{2n+1}$ (and $L=\ccO_{\PP^{2n+1}}(2))$.
  In the first two cases the contact distribution $F\subset TX$ is unique.
  More precisely, for every point $x \in X$, the  subspace $F_x\subset T_x X$ is the smallest linear subspace containing all tangent directions to contact lines in $X$.
\end{prop}
\begin{prf}
  By  \cite[Cor.~4.2]{lebrun_salamon} either $X = \PP(T^* \PP^{n+1})$ or $b_2(X) = 1$. 
  In the latter case,  by Theorem~\ref{thm_classification} either  $\Pic X = \ZZ[L]$, or $X = \PP^{2n+1}$, as claimed.
  The uniqueness for the case with $\Pic X = \ZZ[L]$ is proved in \cite[Cor.~4.5]{kebekus_lines1}.
  The uniqueness for the case $X= \PP(T^* \PP^{n+1})$ follows since we have two projections
     $\PP(T^*\PP^{n+1}) \to \PP^{n+1}$ and $\PP(T^*\PP^{n+1}) \to (\PP^{n+1})^*$ 
     and the tangent spaces to the fibres of both projections are contained in $F$ by \cite[Rem.~2.3]{kebekus_lines1}, 
     or Lemma~\ref{lem_lines_are_integral} below.
\end{prf}

In the case $X=\PP^{2n+1}$, the contact distribution is not unique, and it is determined by a choice of symplectic form on a vector space $\CC^{2n+2}$,
  see \cite[\S E.1]{jabu_dr} for discussion and references.
The two cases of contact Fano manifolds $X= \PP(T^* \PP^{n+1})$ and $X = \PP^{2n+1}$ are relatively well understood,
  in particular the contact lines are all smooth and standard (in case $X= \PP(T^* \PP^{n+1})$) or do not exist (in the case $X = \PP^{2n+1}$, because $L \simeq \ccO_{\PP^{2n+1}}(2)$).
Thus Proposition~\ref{prop_classification_Fano} shows that in the context of our main goals (investigating singular or non-standard lines) 
  it is harmless 
to work in the more restricted setting.
\begin{setting}[contact Fano setting]\label{sett_contact_Pic_generated_by_L}
   In Setting~\ref{sett_contact_projective}, 
      that is $X$ is a projective manifold (over $\CC$) of dimension $2n+1$ with a contact distribution 
      $F \subset TX$ and the quotient line bundle $L=TX /F$, assume in addition that $\Pic X = \ZZ [L]$.
\end{setting}
We will add more notation to this in Setting~\ref{sett_contact_Fano}.

The classification of contact Fano manifolds is known in low dimension.
\begin{thm}\label{thm_classification_in_low_dimension}
   In Setting~\ref{sett_contact_Pic_generated_by_L}, we have
      $n \ge 2$ and if $n=2$, then $X$ is the five dimensional homogeneous $G_2$-manifold.
\end{thm}
The case $n=1$ has been claimed first by Ye \cite{ye}, but his argument contains a gap.
In fact, the cited article only shows Theorem~\ref{thm_classification} in the case $n=1$ and it does not comment on the case $b_2(X)=1$.
However, it is not difficult to treat the missing case, there are at least two approaches.
Firstly, there are few Fano threefolds with Picard number $1$ and index $2$,
  and one can just check all the possibilities.
Alternatively, one can use Hirzebruch-Riemann-Roch theorem and cohomological criterions
  for a manifold to be a projective space.
This latter approach has been implemented in \cite{peternell_jabu_contact_3_folds} in a more general situation.

The case $n=2$ has been proved by Druel \cite{druel} using results of
  \cite{mukai_Fano_3_folds_and_manifolds_of_coindex_3},
  \cite{mella_good_divisors_on_Mukai_varieties},
  and \cite{beauvillefano}.
Further extensions of Theorem~\ref{thm_classification_in_low_dimension} are studied in \cite{jabu_wisniewski_weber_torus_on_contact}.

\subsection{Lines on contact Fano manifolds}\label{sec_history_lines}

\begin{setting}[contact Fano setting in detail]\label{sett_contact_Fano}
   Suppose that $X$ is a complex Fano manifold of dimension $2n+1$ not isomorphic to $\PP^{2n+1}$, nor to $\PP(T^*\PP^{n+1})$.
   Suppose further that there exists a contact distribution $F \subset TX$ on $X$,
     that is we work in Setting~\ref{sett_contact_Pic_generated_by_L}.
   Then $F$ is unique by Proposition~\ref{prop_classification_Fano}.
   Denote by $L= TX/F$ the quotient line bundle,
     and pick $\ccH$ to be any irreducible component of the subset of the Chow variety parametrising contact lines on $X$.
   Finally, let $\ccH_x$ be the set of lines through a fixed point $x\in X$.
   Note that $\ccH$ and $\ccH_x$ are non-empty and $\dim \ccH = 3n-1$ 
      (see \cite[Equation (2.1) and Prop.4.1]{kebekus_lines1}; alternatively, see Corollary~\ref{cor_dim_ccH_eq_3n-1} for details).
   Moreover, every irreducible component of $\ccH_x$ has dimension $n-1$ (\cite[Prop.4.1]{kebekus_lines1} or Proposition~\ref{prop_Hx_Legendrian}).
\end{setting}

We now discuss in more detail the gap in \cite{kebekus_lines2}.
The gap has been previously mentioned  in \cite[Footnote~4 on page~7]{jabu_dr} and in \cite[Rem.~3.2]{jabu_contact_duality_and_integrability}.
In Setting~\ref{sett_contact_Fano} Kebekus claims that for a general point $x \in X$, the set $\ccH_x$ is irreducible \cite[Thm~1.1(2) or Thm~3.1]{kebekus_lines2}.
The argument is presented in \cite[\S3]{kebekus_lines2} and begins with the statement of \cite[Prop.~3.2]{kebekus_lines2}.
Unfortunately, Step 2 of the proof of this proposition is incorrect. 
It constructs a family of varieties denoted $V \to D^0$, where $V \subset D^0 \times X$.
(In this article $D^0$ is denoted by $B$, see Theorem~\ref{thm_dimension_of_the_space_of_special_lines}, or Section~\ref{sect_divisors_of_non_standard_lines}.)
The family is of pure dimension $n$, and Kebekus claims that $V \to D^0$ is 
``a well-defined family of cycles in $X$ in the sense of \cite[Chapt.~I.3.10]{kollar_book_rational_curves}''.
By ``Chapt.~I.3.10'' Kebekus means Definition~3.10 in Chapter~I of the book.
This definition has four items. 
It is indeed straightforward to verify that the first three items (3.10.1)--(3.10.3) are satisfied in the situation of $V \to D^0$.
However, the final one (3.10.4) very roughly says that for every $w\in D^0$, and a curve containing $w$ as a special point,
  the limit of the cycle over a generic point of the curve is equal to the cycle over $w$ (see Comment (3.10.5)).
A simple situation when this property fails is the family of $0$-dimensional cycles in $\PP^1$, 
  with the base of the family parametrised by a nodal rational curve $C$.
Then we let the universal family $U$ be $\PP^1 \subset C \times \PP^1$, with the first projection equal to the normalisation map and the second projection equal to the identity map.
Such a family satisfies (3.10.1)--(3.10.3), but fails to satisfy (3.10.4), because near the singular point the limit is either one or the other point in the normalisation.
This last property (3.10.4) is not justified in the proof of \cite[Step~2, Prop.~3.2]{kebekus_lines2}, and it is not at all obvious.

In \cite[Thm~I.3.17]{kollar_book_rational_curves} it is shown that if the base is normal, then the property (3.10.4) is automatic.
Thus the claim of \cite[Step~2, Prop.~3.2]{kebekus_lines2} is fine if $D^0$ is assumed in addition to be normal.
This observation, together with Lemma~\ref{lem_ccB_x_is_a_union_of_irreducible_components}, proves Theorem~\ref{thm_dimension_of_the_space_of_special_lines}\ref{item_B_nonnormal}.
One could hope to fix the issue by replacing $D^0$ with its normalisation. 
However, the normalisation of $D^0$ needs not to have $b_2=1$.
For instance, there exists an example of a non-normal cubic hypersurface in $\PP^4$,  
  whose normalisation is $\PP^1 \times \PP^2$, thus has Picard rank $2$, see \cite[Lem.~2.2(1)c, Thm~3.1(a.5)]{lee_park_schenzel_nonnormal_cubics}).
Consequently, the further part of the proof of Step 2 fails: the induced map to the Chow variety needs not to be finite.
In fact, one consequence of Theorem~\ref{thm_dimension_of_the_space_of_special_lines} 
  is that all the fibres must have dimension $n$.

\subsection{Sketch of proofs and intermediate results}\label{sect_sketch}

The arguments for Theorem~\ref{thm_singular_lines_on_non_projective} 
  rely on a technical Proposition~\ref{prop_singular_lines}, 
  whose proof follows the method of \cite[Prop.~3.3]{kebekus_lines1}, carefully adapted to the setting of quasi-projective varieties and analytic varieties.
Proposition~\ref{prop_line_through_general_point_is_standard_on_non_projective} 
  is proved by a standard analysis of possible splitting types of the tangent bundle restricted to lines.
See Corollary~\ref{cor_non_standard_lines_do_not_cover_X} for a stronger version of this statement.

The proof of Theorem~\ref{thm_dimension_of_the_space_of_special_lines} is more tricky. 
It is centred around the following concept:
\begin{defin}\label{defin_linear_subspace}
   Suppose that $X$ is a complex manifold $X$ and $L$ is a line bundle on $X$. 
   A \emph{linear subspace} in $(X,L)$ is
      a compact subvariety $\Gamma \subset X$ 
      such that the normalisation of $\Gamma$ is a projective space 
      and the restriction of $L$ to $\Gamma$ is a line bundle of degree $1$,
  see \S\ref{sect_lines_and_linear_subspaces} for more details.
  A linear subspace of dimension $1$ is a called a \emph{line}.
\end{defin}
    In particular, this concept generalises the notion of contact line 
    to higher dimensions (and to more general varieties).
In the situation of the definition we say that two points $x, y\in X$ are 
    \emph{connected by a line},
   if there exists a single line $C \subset X$, such that $x,y \in C$.
Clearly, any two points on a linear subspace $\Gamma \subset X$ of any dimension are connected by a line.
In fact, this property characterises linear subspaces.

\begin{thm}\label{thm_many_lines}
   Suppose that $\Gamma$ is a projective variety with an ample line bundle $L$,
     such that two general points $x,y \in \Gamma$ are connected by a single line.
   Then $\Gamma$ is a linear subspace of itself, that is, it admits a normalisation $\mu\colon \PP^k \to \Gamma$,
     where $k=\dim \Gamma$ and $\mu^*L = \ccO_{\PP^k}(1)$.
\end{thm}
This theorem is an easy consequence of a characterisation of projective spaces by Kebekus
  \cite[Thm~3.6]{kebekus_families_of_singular_rational_curves},
  taking into account that in our case $\Gamma$ may be not normal.

\begin{proof}
   First let us reduce to the case in which $\Gamma$ is normal.
   So let $\mu\colon \Gamma' \to \Gamma$ be the normalisation,
      and pick two general points $x',y' \in \Gamma'$.
   Their images $\mu(x')$ and $\mu(y')$ are connected by a line $C \subset \Gamma$.
   Let $C'$ be the proper transform of $C$.
   Then $\mu|_{C'} \colon C' \to C$ is birational and $C$ is also a rational curve, 
      and the normalisation of $f  \colon \PP^1 \to C$ factorises through $C'$.
   So the degree of $C'$ with respect to $\mu^* L$ is also $1$.
   Thus it is sufficient to prove the theorem for normal $\Gamma$.

   The set of lines in the Chow variety of $\Gamma$ is closed by
   \cite[Thm~I.3.21]{kollar_book_rational_curves} 
   (or Proposition~\ref{prop_linear_spaces_and_Chow_variety_for_projective}\ref{item_linear_spaces_are_closed} and the comment following its proof for more details),
      therefore also special points $x$ and $y$ are connected by curves of degree $1$.
   Then this is a special case of
      \cite[Thm~3.6]{kebekus_families_of_singular_rational_curves}.
\end{proof}

To show Theorem~\ref{thm_dimension_of_the_space_of_special_lines},
  we suppose that there is a component $\ccB$ of the set of non-standard lines of dimension $3n-2$.
In particular, 
  a general element of $\ccB$ is a smooth rational curve (Proposition~\ref{prop_dimension_of_the_space_of_singular_lines}).
By results of Kebekus, there is a divisor $B$ on $X$ swept out by the lines from $\ccB$ (see Lemma~\ref{lem_ccB_x_is_a_union_of_irreducible_components}).
The main aim is to prove that $B$ is dominated by a family of linear subspaces of dimension $n$.
To construct these linear subspaces we use 
   Theorem~\ref{thm_many_lines} twice.
In the first place we construct a large family of linear subspaces of dimension $2$, that is planes.
Next we bundle together these planes, to obtain a family of linear subspaces of dimension $n$.
More precisely, the tangent spaces to $B$ naturally determine a distribution $G$ of rank $1$,
  that is, a line subbundle of $TB$ defined on an open dense subset of $U \subset B$.
Suppose that $c \in \ccB$ is a general non-standard line and $C \subset X$ is the corresponding curve in $X$.
We take the union of leaves of $G$ through points of $C$, and 
we let $\overline{\Gamma_c}$
   to be the Zariski closure of this union.
Then we show in Lemma~\ref{lemma_any_points_on_Gamma_c_are_connected_by_a_line}
   that every two points in  $\overline{\Gamma_c}$ are connected by a contact line.
Thus the normalisation of  $\overline{\Gamma_c}$  is a projective space $\PP^k$.
We carefully study the distribution $G$ restricted to $\overline{\Gamma_c}$
   and conclude using Lemma~\ref{lem_distribution_on_Pk} that the leaves of $G$ actually are lines from $\ccB$.
In particular, $\dim \overline{\Gamma_c} = 2$, and its normalisation is $\PP^2$.

This construction also equips each $\PP^2$ with a distinguished point $y$, and its image in $X$.
We consider $Y \subset X$ to be the union of all the distinguished points in $X$ obtained by varying $c \in \ccB$.
The critical step in the proof is the dimension count: we show that $\dim Y =n$, see Lemma~\ref{lem_dim_Y_eq_n}.
The conclusion is that there are many surfaces $\overline{\Gamma_c}$ with the same distinguished point $y$.
On the other hand the locus $P^y$ of these projective planes is always contained in the locus swept by lines through
  a fixed point $y$, which is known to have dimension at most $n$.
We use this information to show that two general points $x_1, x_2$ in $P^y$
  are contained in a single $\overline{\Gamma_c}$, whose distinguished point is $y$.
The line in the plane $\PP^2$ normalising $\overline{\Gamma_c}$ is the required line connecting $x_1$ with $x_2$.
Thus $P^y$ is normalised by a projective space, and its dimension is calculated to be $n$.
This is the way to construct the family of linear subspaces of dimension $n$,
  whose locus sweeps out the divisor $B$. 
  
We also show that there is exactly one such linear space through a general point of $B$, 
  and only a finite number of them through any point of $B$. 
This is used to compare the family to the normalisation of $B$.
Finally, we conclude using \cite[Prop.~4.10]{araujo_druel_codim1_del_Pezzo_foliations}
  or Proposition~\ref{prop_linear_spaces_and_Chow_variety_for_projective}\ref{item_families_of_linear_spaces_are_bundles},
  which characterises projective space bundles over normal varieties,
  by analogy to Fujita's characterisations for bundles on smooth varieties.
\section{Parameter spaces for lines and linear subspaces}\label{sect_preliminaries}

Throughout this article we suppose that $X$ is a complex analytic variety with a distinguished line bundle $L$. 
We will say $L$ is a \emph{polarisation} of $X$, as it will be used to measure degrees of some special rational subvarieties.
Although we will often assume $X$ is a projective manifold and $L$ is ample,
  for some of the statements below it is not necessary.
In particular, we do not want to assume $X$ is compact or nonsingular to obtain valid statements for open subsets of $X$, and subvarieties 
    $Y\subset X$ polarised by $L|_Y$.

\subsection{Barlet space and Chow variety}
\label{sect_barlet_and_chow}

In this subsection we overview the properties of the \emph{cycle space}, which in the context of complex geometry is called the 
\emph{Barlet space} \cite{barlet_espace_des_cycles},
while in algebraic geometry it is called the 
\emph{Chow variety} \cite[Section~I.3]{kollar_book_rational_curves}.
We are only interested in compact cycles, and in fact mainly in irreducible ones. The reducible ones (that is, either those with more than one components, or those with multiplicity higher than $1$ at one of the components) appear in our considerations mainly as limits of irreducible cycles (which are just compact subvarieties).
See also \cite[Chapter~IV]{barlet_magnusson_cycles_analitiques_I} or \cite[Chapter VIII, \S2]{grauert_peternell_remmert_several_complex_variables_VII}.

Throughout our proofs we exploit some minimal assumptions on the cycle spaces 
  (and also some other parameter spaces),
  so that the local geometry 
  of (for instance)  singular rational curves determines its global properties.
As an example, we want the dimension of the closure of the locus swept out by the singular rational curves to be determined by
  their infinitesimal deformations.
This is guaranteed if either $X$ is quasi-projective or $X$ is compact of Fujiki class $\ccC$.

Note that throughout this section, whenever we speak about \emph{pull-back} in the context of cycles, we always mean the \emph{Chow pull-back} 
  \cite[Def.~I.3.18]{kollar_book_rational_curves}, or \emph{cycle-theoretic base change} \cite[Def.~IV.3.1.1]{barlet_espace_des_cycles}.
Similarly \emph{fibre} refers to the cycle theoretic notion.
The main difference between the cycle theoretic and scheme theoretic pull-backs and fibres is that we ignore the scheme structure, 
  but remember the multiplicity of components instead.

We summarise a list of properties of the Chow variety and Barlet space that we are going to freely use.
The main purpose of the lengthy propositions below is to clarify the meaning of, for example, ``the set of singular rational curves of degree $1$ with respect to a line bundle'' and their irreducible components.
In particular, the proposition illustrates that these irreducible components have sensible structures of analytic or algebraic varieties. 
We restrict our attention to a compact analytic space of Fujiki class $\ccC$ 
   (that is, it is compact and bimeromorphically equivalent to a compact K\"ahler manifold) 
   and quasi-projective varieties.

We say a cycle is \emph{reducible} (respectively, \emph{irreducible}), 
  if  it consists of at least two (respectively, exactly one) components counted with their multiplicity. 
Note that the closed reduced analytic spaces 
  or algebraic subschemes appearing in 
  items \ref{item_trees_are_closed}--\ref{item_singularities_are_closed} 
  below can be reducible thus it is not correct to say 
  that they are varieties.

\begin{prop}\label{prop_Chow_and_Barlet}
  Suppose that $X$ and $\ccR$ are irreducible analytic varieties (respectively, algebraic varieties) and $L$ is a line bundle on $X$. 
  \begin{interruptedenumerate}
  \begin{insideenumerate}
    \item  \label{item_chow_pullback}
           If there is a holomorphic (respectively, algebraic) map from $\ccR$ to the Barlet space (respectively, Chow variety) of $X$,
           then there is $\ccU_{\ccR} \subset X\times \ccR$ consisting of the proper cycles represented by points in the image of points in $\ccR$.
           The map $\ccU_{\ccR} \to \ccR$ is proper.
    \item  \label{item_chow_functoriality} 
           Conversely, if $\ccU_{\ccR} \subset X\times \ccR$ is an analytic (or well defined) family of (proper) cycles in the sense of  \cite[Chapter VIII, Def.~2.5]{grauert_peternell_remmert_several_complex_variables_VII} 
           or \cite[I.3.10]{kollar_book_rational_curves},
           then there is a map from $\ccR$ to the Barlet space (respectively, Chow variety), such that $\ccU_{\ccR}$ is the  family constructed in \ref{item_chow_pullback}.
    \item  \label{item_the_degree_is_constant_on_components}
           In the situation of \ref{item_chow_pullback} and \ref{item_chow_functoriality} the dimension and degree with respect to $L$ of the cycles represented by the points of $\ccR$ are constant. 
  \end{insideenumerate}
  Furthermore, suppose that $X$ is compact of Fujiki class $\ccC$ (respectively, projective). 
  \begin{insideenumerate}
    \item  \label{item_fujiki_class_C_has_Barlet_compact}
           If $X$ is a compact manifold of Fujiki class $\ccC$, then the connected components of the Barlet space are compact of Fujiki class $\ccC$. 
    \item  \label{item_projective_has_Chow_projective}
           If $X$ is projective, then the connected components of the Chow variety are projective.
    \item  \label{item_trees_are_closed} 
           The set of reducible cycles 
              in every connected component of the Barlet space (respectively, Chow variety)
              is a closed reduced analytic subspace (respectively, a closed reduced algebraic subscheme) of that component.
           Thus the set of irreducible cycles is open in the cycle space, and in each irreducible component of the cycle space, 
              this open subset is either empty or dense.
    \item  \label{item_ratcurves_and_trees_closed}
           The set consisting of proper rational curves is a closed reduced analytic subspace (respectively, closed reduced algebraic subscheme)
              in each connected component of the set of irreducible cycles.
           Moreover, the closure of the set of proper irreducible rational curves in the cycle space is a closed reduced analytic subspace  (respectively, closed reduced algebraic subscheme).
    \item  \label{item_singularities_are_closed}
           The set of singular rational curves is closed  reduced analytic subspace (respectively, closed reduced algebraic subscheme) in each connected component of the set of proper rational curves.
  \end{insideenumerate}
  \end{interruptedenumerate}
\end{prop}

\begin{prf}
    Items \ref{item_chow_pullback} and \ref{item_chow_functoriality} are the merits of the definitions of cycle spaces, see for instance
          \cite[Chapter VIII, Thm~2.7]{grauert_peternell_remmert_several_complex_variables_VII}  or \cite[Thm~I.3.21]{kollar_book_rational_curves}.
    Further, the dimension part in \ref{item_the_degree_is_constant_on_components} is again part of the definition of the cycle space, 
          while the degree part is shown in \cite[Prop.~I.3.12]{kollar_book_rational_curves} for the algebraic case.
    For the analytic case, we imitate the algebraic proof.
    It is enough to argue locally on $\ccR$ and it is enough to consider the case when $\ccR$ is $1$-dimensional. 
    We may also normalise $\ccR$ and hence assume $\ccR$ is a holomorphic disc in $\CC$ around $0$.
    Now, with these assumptions, the components of $\ccU_{\ccR}$ are flat over $\ccR$.
    For flat maps the degree is preserved along the fibres and summing over the components we obtain the statement.

    Item~\ref{item_fujiki_class_C_has_Barlet_compact} is quoted in \cite[Chapter VIII, Prop.~3.17]{grauert_peternell_remmert_several_complex_variables_VII}, 
       and the details are attributed to \cite{fujiki_closedness_of_Douady_spaces}.
    Analogously, \ref{item_projective_has_Chow_projective} is shown in \cite[Thm~I.3.21.3]{kollar_book_rational_curves}. 
          
    Item~\ref{item_trees_are_closed} is an immediate conclusion from \cite[Prop.~IV.7.1.2]{barlet_magnusson_cycles_analitiques_I}.
    
    The set of rational $1$-dimensional cycles (in the sense of \cite{barlet_rational_1_dimensional_compact_cycles}, that is, those cycles, whose every component is a rational curve) forms a closed analytic subset of the cycle space by \cite{barlet_rational_1_dimensional_compact_cycles} 
       (which is projective by \ref{item_projective_has_Chow_projective}, if $X$ is projective).
    Irreducible rational curves form a Zariski open subset in the set of rational $1$-dimensional cycles by \ref{item_trees_are_closed}, which shows both claims of 
    item~\ref{item_ratcurves_and_trees_closed}.

    To show \ref{item_singularities_are_closed} suppose that $\ccR$ is the set of proper rational curves in the cycle space and $\ccU_{\ccR}$ is the universal family as in \ref{item_chow_pullback}.
    Let $S \subset \ccU_{\ccR}$ be the set of singular points of the fibres. It is a closed analytic or algebraic subspace as it is locally the zero locus of some minors of differentials.
    Its image under the proper map to the cycle space is again closed and analytic by Remmert's mapping theorem \cite[Chapter~III, Cor.~4.3]{grauert_peternell_remmert_several_complex_variables_VII} (in the algebraic geometry setup, it follows directly from the definitions of a proper morphism \cite[p.~100]{hartshorne} and the Zariski topology).
\end{prf}

We are also going to use the analogous properties for a quasi-projective variety $X$ with a projective compactification $\overline{X}$.
In this situation, the Chow variety of $X$ is the open subset of the Chow variety of $\overline{X}$ consisting of those cycles that do not intersect the boundary $\overline{X} \setminus X$.
Hence the proofs of the properties below boil down to just applying the projective case and restricting to the open subset.
The numeration in the statements below corresponds to the analogous statements in Proposition~\ref{prop_Chow_and_Barlet}. 
As before, the purpose of the proposition is to give a sensible structure of a variety to irreducible components of various sets parametrising singular or nonsingular cycles, 
   reducible or irreducible cycles, or rational curves, perhaps with prescribed additional properties.
\begin{prop}
  Suppose that $X$ is a quasi-projective variety. 
  \begin{enumerate}
    \setcounter{enumi}{4}        
    \item  \label{item_quasi-projective_has_Chow_quasi-projective}
           The connected components of the Chow variety of $X$ are themselves quasi-projective.
    \item  
           The set of reducible cycles 
              in every connected component of the Chow variety is Zariski closed, in particular it is a closed reduced algebraic subscheme of that component.
           Thus the set of irreducible cycles is open in the cycle space, and in each irreducible component of the cycle space, either empty or dense.
    \item  
           The set consisting of proper rational curves is Zariski closed  in each component of the set of irreducible cycles, 
             in particular, it is a reduced algebraic scheme.
    \item  
           The set of singular rational curves is Zariski closed  in each component of the set of proper rational curves.
           In particular, it is a reduced algebraic scheme.
  \end{enumerate}
\end{prop}
\noprf

Thus if $X$ is compact of Fujiki class $\ccC$ or quasi-projective, then for instance the set of irreducible singular rational curves is a union of its irreducible components. 
Say one of these components is $\ccS$ and suppose that $\overline{\ccY}$ is an irreducible component of the cycle space of $X$ 
   (if $X$ is compact of Fujiki class $\ccC$ or projective; in this case denote $\overline{X}:=X$) or of $\overline{X}$ (if $X$ is quasi-projective) containing $\ccS$.
Let $\ccY \subset \overline{\ccY}$ be the open dense subset of irreducible cycles contained in $X$.
Then $\ccS \subset \ccY$ and $\ccS$ is a closed analytic (respectively, algebraic) subvariety in $\ccY$.
Let $\overline{\ccS}$ be its closure in $\overline{\ccY}$ which is an irreducible analytic (respectively, algebraic) subvariety in a $\overline{\ccY}$.
   
The \emph{locus} $\overline{S} \subset \overline{X}$ swept by $\overline{\ccS}$ is the union of cycles in $\overline{\ccS}$,
   which also can be obtained as the image of the universal family $\ccU_{\overline{\ccS}}$ as in
   Proposition~\ref{prop_Chow_and_Barlet}\ref{item_chow_pullback}. 
By Remmert's mapping theorem it is a closed analytic subvariety (or closed algebraic subvariety).
Note that it is irreducible by construction.

Similarly, the \emph{locus} $S \subset X$ swept by $\ccS$ is the union of cycles in $\ccS$, or the image of the universal family $\ccU_{\ccS}$ over $\ccS$.
Since $\ccU_{\ccS}\subset \ccU_{\overline{\ccS}}$ is an open subset, a complement of a closed analytic subspace (respectively, algebraic subscheme),
the image $S$ is constructible by Chevalley-Remmert Theorem \cite[Thm~on~p291 and Rem.~on~p293]{lojasiewicz_intro_to_complex_analytic_geometry}.
In particular, $S$ is dense in $\overline{S}$ and contains an open subset also dense in $\overline{S}$.
So it makes sense to speak about the tangent space $T_s S$ to $S$ at a general point $s \in S$ and by Sard Theorem $T_s S$ 
   is the image of a tangent space to $\ccU_{\ccS}$ at a (general) point $ u \in \ccU_{\ccS}$ such that $u\mapsto s$.

Analogously, we may define the loci of other sensible families of subvarieties, such as; 
   proper rational curves through a fixed point $x\in X$, 
   rational curves with a non-standard splitting type of the tangent bundle (see \S\ref{sect_splitting_types}),
   linear subspaces (see \S\ref{sect_lines_and_linear_subspaces}).
The sensibility of a family of cycles $\ccR$ is guaranteed if its closure $\overline{\ccR}$ 
   in the cycle space of $\overline{X}$ is a closed reduced analytic subspace (respectively, closed reduced algebraic subscheme) 
   and also the boundary $\overline{\ccR}\setminus \ccR$ is a closed analytic subspace (respectively, closed algebraic subscheme).
In Corollary~\ref{cor_locus_of_lines_has_a_submanifold} we also show that 
   it is sometimes sensible to consider the loci of ``small'' families of curves.

\subsection{Lines and linear subspaces of a polarised analytic set}
\label{sect_lines_and_linear_subspaces}

Recall from Definition~\ref{defin_linear_subspace} the notions of line and linear subspace of a polarised complex analytic space.

\begin{defin}
A \emph{family of linear subspaces of dimension $k$} is:
\begin{itemize}
 \item a proper surjective morphism $\pi\colon \ccU_{\ccR} \to \ccR$ between reduced analytic spaces (respectively, algebraic varieties) 
          with all the (cycle theoretic) fibres isomorphic to $\PP^k$, and
 \item a map $\xi \colon \ccU_{\ccR} \to X$, 
\end{itemize}
such that all the images in $X$ of fibres of $\pi$ are linear subspaces and $\xi|_{\PP^k}\colon \PP^k \to X$ is the normalisation map of the image.
The map $\xi$ is called the \emph{evaluation map}. 
If $k=1$, we simply say a \emph{family of lines},  rather than a family of linear subspaces of dimension $1$.
 \end{defin}

If necessary, we may always replace $\ccR$ with its normalisation, and $\ccU_{\ccR}$ with the pullback, so it is harmless to assume that $\ccR$ is normal. 
Usually, we may also assume that $\ccU_{\ccR}$ is normal, as shown in Lemma~\ref{lem_can_take_U_normal}.
Solely for the purpose of this lemma we define a condition $(\star)$.
We say a morphism $\pi\colon \ccU \to \ccR$ of analytic spaces satisfies $(\star)$, if:
   \begin{itemize}
    \item[$(\star)$] each set theoretic fibre of $\pi$ is $\PP^k$ and there exists a line bundle on $\ccU$, which restricts to $\ccO(1)$ on each (set-theoretic) fibre.
   \end{itemize}

\begin{lemma}\label{lem_can_take_U_normal}
   Suppose that $\pi\colon \ccU \to \ccR$ is a morphism of analytic spaces that satisfies $(\star)$ and $\ccR$ is normal. 
   Let $\ccU^{norm}$ be the normalisation of $\ccU$.
   Then the composed map $\pi^{norm} \colon \ccU^{norm} \to \ccR$ satisfies $(\star)$.
\end{lemma}
\begin{prf}
   The normalisation map $\ccU^{norm} \to \ccU$ is finite and birational, 
      in particular it is finite and birational when restricted to a general fibre of $\pi^{norm}$ and $\pi$.
   Thus, it is an isomorphism of the general, reduced fibres.
   Let $L$ be the line bundle on $\ccU$, that restricted to the fibres is $\ccO(1)$,
      and $L^{norm}$ be its pullback to $\ccU^{norm}$.
   Then $L^{norm}$ has degree $1$ on a general fibre and thus on every fibre.
   Therefore $\pi^{norm}$ is birational on every fibre, and all the set-theoretic fibres are $\PP^k$.
   Furthermore, $L^{norm}$ restricts to $\ccO(1)$ on every fibre.
\end{prf}

The following lemma shows that a general fibre of a morphism of normal varieties is necessarily normal.

\begin{lemma}\label{lem_general_fibre_normal}
   Let $f \colon Y \to Z$ be a proper morphism of irreducible normal analytic varieties (or irreducible normal complex algebraic varieties). 
   Then there is an open dense subset $U\subset Z$ such that for all $u\in U$ the fibre $f^{-1}(u)$ is normal.
   Moreover, $U$ may be chosen as a complement of a closed analytic subset of $Z$, respectively, of a closed algebraic subvariety.
\end{lemma}

\begin{prf}
   By Hironaka's Flattening Theorem \cite[Cor.~1]{hironaka_flattening_theorem_in_complex_analytic_geometry} we may assume in addition that $f$ is flat,
     because the flattening procedure does not change the general fibre.
   Let $N(f) := \set{p \in Y \colon Y_{f(p)} \text{ is normal at $p$}} \subset Y$, where $Y_{f(p)}$ is the fibre of $f$ over $f(p)$.
   The set $N(f)$ is a (closed) analytic subset of $Y$ by \cite[Prop.~3.22 on p.~160]{fischer_complex_analytic_geometry}, 
      or a (closed) algebraic subvariety of $Y$ by \cite[Thm~12.1.6(iv)]{ega4_3}. 
   By Remmert's mapping theorem \cite[Chapter~III, Cor.~4.3]{grauert_peternell_remmert_several_complex_variables_VII} 
     the image $ f(N(f))$ is a closed analytic subset of $Z$. 
   In the algebraic case $f(N(f))$ is a closed algebraic subset of $Z$ by the definition of proper morphism, see for instance \cite[Chapter~II.4]{hartshorne}.
   It remains to prove that $f(N(f)) \ne Z$.

   Since our base field $\CC$ is algebraically closed, ``normal'' is the same as ``geometrically normal''.
   We are going to reduce the analytic setting to the local algebraic setting using completions.
   Without loss of generality, $Z$ is smooth (by replacing $Z$ with its smooth locus).
   Further, we may argue locally and replace $Z$ with a local complete neighbourhood of a smooth point $z \in Z$ and $Y$ with its preimage.
   That is $Z \simeq \Spec\CC[[\fromto{z_1}{z_n}]]$.
   Further, since normality is a local property, we may replace $Y$ with a completion of the local ring of a closed point in $y \in Y$.
   
   Thus we have a dominant morphism $Y \to Z$ of local Noetherian schemes that corresponds to an inclusion of algebras 
     $\CC[[\fromto{z_1}{z_n}]] \to \ccO_Y(Y)$.
   The claim of the lemma is that the preimage $Y_{\eta}$ of the generic point $\eta\in Z$ is normal.
   Indeed, $Y_{\eta} = \Spec (\CC[[\fromto{z_1}{z_n}]] \setminus \{0\})^{-1}\ccO_Y(Y)$, 
      which is normal, as it is a spectrum of a multiplicative set times an integrally closed ring.
\end{prf}

The following proposition explains the relation between the families of linear subspaces and the corresponding subvarieties of the Barlet space (respectively, the Chow variety).

\begin{prop}\label{prop_linear_spaces_and_Chow_variety}
   Let $X$ be a complex analytic variety (respectively, complex algebraic variety) with a line bundle $L$.
   \begin{enumerate}
    \item  \label{item_families_of_linear_spaces_determine_morphism_to_Chow}
   Suppose that $\pi\colon \ccU_{\ccR} \to \ccR$ is a family of linear subspaces of $(X,L)$ with evaluation map $\xi \colon \ccU_{\ccR} \to X$.
   Suppose in addition that $\ccR$ is normal. 
   Then there is a morphism from $\ccR$ to the Barlet space of $X$ (respectively, to the Chow variety of $X$), 
       such that for all points $r\in \ccR$, the image of $r$ in the Barlet space (respectively, in the Chow variety) of $X$ 
       is the point representing the linear subspace $\xi(\pi^{-1}(r))$.
    \item    \label{item_families_of_linear_spaces_from_Chow_variety_on_open} Suppose that $\ccR'$ is an irreducible analytic subvariety of the Barlet space (respectively, the Chow variety) of $X$, 
     whose all elements represent linear subspaces of $(X,L)$ of fixed dimension $k$. 
   Then there exists a family of linear subspaces $\pi\colon \ccU_{\ccR} \to \ccR$, such that $\ccR \subset \ccR'$ is an open dense subset and for each $r\in \ccR$,
     the image $\xi(\pi^{-1}(r))$ is the linear subspace corresponding to the point $r \in \ccR'$ of the Barlet space (respectively, Chow variety).
   \end{enumerate}
\end{prop}
\begin{prf}
   To see the first item, 
      consider the incidence subvariety $\ccU' \subset \ccR \times X$, $\ccU' := \set{(r,x) \mid x\in \xi(\pi^{-1}(r))}$, 
      that is, $\ccU' = (\pi \times \xi) (\ccU_{\ccR})$. Note that the projection $\ccU' \to \ccR$ is proper and $\ccU'$ is reduced as it is an image of reduced $\ccU$.
   Hence $\ccU' \to \ccR$ is a well defined family of cycles in the sense of \cite[Def.~1.3.10]{kollar_book_rational_curves} 
      (note that property (1.3.10.4) is implied by the normality and Theorem~1.3.17 in the same book), 
      or analytic family of cycles in the sense of \cite[Chapter VIII, Def.~2.5]{grauert_peternell_remmert_several_complex_variables_VII}. 
   Thus by the universal property of Chow variety or Barlet space (Proposition~\ref{prop_Chow_and_Barlet}\ref{item_chow_functoriality}),
      there is a morphism from $\ccR$ to the cycle space of $X$ satisfying the required properties.
   
   To see the second item, we may assume that $\ccR'$ is smooth by restricting to an open dense subset. 
   Let $\ccU'$ be the normalisation of the universal family of the Barlet space (respectively, Chow variety) restricted to $\ccR'$,
     so that $\ccU' \to \ccR'$ is a proper surjective morphism with fibres mapped onto linear subspaces of $X$.
   Since both $\ccU'$ and $\ccR'$ are normal, the general fibre is also normal by Lemma~\ref{lem_general_fibre_normal}.
   Define $\ccR \subset \ccR'$ to be an open dense subset containing only points with normal fibre, 
     and $\ccU$ is the restriction of $\ccU'$ to $\ccR$. 
   The map $\ccU \to \ccR$ is proper, since it is a base change of a proper map $\ccU' \to \ccR'$.
   Let $r \in \ccR$ and $\ccU_r$ be the fibre.
   Then the evaluation map $\ccU_r \to X$, whose image is the linear subspace corresponding to $r$, is finite and birational. 
   Since $\ccU_r$ is normal, it must be the normalisation map and $\ccU_r \simeq \PP^k$. 
   This shows that $\ccU \to \ccR$ is a family of linear subspaces. 
\end{prf}

For the rest of this subsection we suppose that $X$ is in addition projective and $L$ is ample.
In this setting we may strengthen Proposition~\ref{prop_linear_spaces_and_Chow_variety} using \cite[Prop.~4.10]{araujo_druel_codim1_del_Pezzo_foliations}:
\begin{prop}\label{prop_linear_spaces_and_Chow_variety_for_projective}
   Let $X$ be a projective variety with a line bundle $L$.
   \begin{enumerate}
      \item \label{item_families_of_linear_spaces_are_bundles}
              Suppose that $\pi\colon \ccU_{\ccR}\to \ccR$ is a family of linear subspaces of $(X,L)$ of dimension $k$ with $\ccR$ and $\ccU_{\ccR}$ normal
                 and with evaluation map $\xi\colon\ccU_{\ccR} \to X$.
              Let $\ccE := \pi_*(\xi^*L)$. 
              Then $\ccE$ is a vector bundle of rank $k+1$ on $\ccR$ and $\ccU_{\ccR} \simeq \PP(\ccE^*)$ with $\ccO_{\PP(\ccE^*)}(1) = \xi^*L$.
      \item \label{item_families_of_linear_spaces_from_Chow_variety} 
             Suppose that $L$ is ample and $\ccR$ is an irreducible normal variety with a morphism to the Chow variety of $X$,
             such that all points in the image represent linear subspaces of $(X,L)$ of a fixed dimension $k$.
             Then there exists a family of linear subspaces $\pi\colon \ccU_{\ccR} \to \ccR$ with evaluation map $\xi$, 
                such that for each $r\in \ccR$, the image $\xi(\pi^{-1}(r))$ is the linear subspace corresponding to the image of $r$ in the Chow variety.
      \item \label{item_linear_spaces_are_closed}
             Suppose that $L$ is ample. The set of linear subspaces of dimension $k$ is a Zariski closed subset of the Chow variety of $X$. 
   \end{enumerate}
\end{prop}

In the following proof and in the further sections of this article we will use the \emph{$\Hom$-scheme} $\Hom(Y\to X)$.
It is the scheme that parametrises all morphisms $Y \to X$ 
between two varieties or schemes.
Formally, its definition uses \cite[Def.~I.1.9]{kollar_book_rational_curves}. 
The construction of $\Hom(Y\to X)$ as an open subscheme of the Hilbert scheme of $Y\times X$ is presented in 
\cite[Thm~I.1.10]{kollar_book_rational_curves} (to each morphism one associates its graph).
See also \cite[\S4.6.6]{sernesi_Deformations_of_algebraic_schemes}
For the existence of $\Hom$-scheme in the setting of analytic spaces see  \cite[Thm~VIII.1.5]{grauert_peternell_remmert_several_complex_variables_VII} 
or \cite[\S10.2, Thme 1]{douady_Douady_spaces}.

\begin{prf}
   To prove \ref{item_families_of_linear_spaces_are_bundles} 
   observe that $\xi^*L$ is ample on every scheme theoretic fibre of $\pi$, since it is $\ccO_{\PP^k}(1)$ on every set theoretic fibre and ampleness does not depend on non-reduced part of the scheme structure \cite[Prop.~1.2.16(i)]{lazarsfeld}.
   Thus $\xi^*L$ is $\pi$-ample by \cite[Thm~1.7.8]{lazarsfeld} 
      and $\pi$ is equidimensional.
   The claim of \ref{item_families_of_linear_spaces_are_bundles} follows from 
      \cite[Prop.~4.10]{araujo_druel_codim1_del_Pezzo_foliations}.
   
   Now we show \ref{item_families_of_linear_spaces_from_Chow_variety}.
   Let $\ccU_{\ccR}$ be the normalisation of the pullback of the universal family from the Chow variety.
   Let $\xi\colon \ccU_{\ccR} \to X$ be the composed map and $\pi\colon \ccU_{\ccR}\to \ccR$ be the projection.
   We continue as in the proof of \ref{item_families_of_linear_spaces_are_bundles}: 
      $\xi^*L$ is ample on all fibres of $\pi$, since $\xi$ restricted to such fibre is finite. 
   Thus $\xi^*L$ is $\pi$-ample \cite[Thm~1.7.8]{lazarsfeld} and equidimensional, with general fibre $\PP^k$ by 
      Proposition~\ref{prop_linear_spaces_and_Chow_variety}\ref{item_families_of_linear_spaces_from_Chow_variety_on_open},
      and $\xi^*L|_{\PP^k} \simeq \ccO(1)$,
      thus by \cite[Prop.~4.10]{araujo_druel_codim1_del_Pezzo_foliations} all fibres are $\PP^k$ and $\pi$ is a family of linear subspaces. 
      
   Finally, \ref{item_linear_spaces_are_closed} is also similar:
   let $\ccR'$ be the subset of the Chow variety consisting of linear spaces of dimension $k$.
   By \cite[Thm~I.3.21]{kollar_book_rational_curves} the set $\ccR'$ is contained in a projective reduced scheme parametrising cycles of degree $1$.
   We have to prove that $\ccR'$ is Zariski closed.
   First we show that it is constructible. 
   Indeed, consider the $\Hom$-scheme $\Hom (\PP^k \to X)$ and inside this $\Hom$-scheme the reduced subscheme of \emph{linear} homomorphisms $\Hom^{lin}(\PP^k \to X)$.
   This subscheme consists of the points representing homomorphisms $\phi \colon \PP^k \to X$ such that $\phi^* L \simeq \ccO_{\PP^k}(1)$.
   Let $\Hom^{lin, n}$ be the normalisation of $\Hom^{lin}(\PP^k \to X)$. 
   Then the projection map $\Hom^{lin,n} \times \PP^k \to \Hom^{lin,n}$ together with evaluation map   
     $\Hom^{lin,n}\times \PP^k \to  X$ is a family of linear subspaces of $X$.
   By Proposition~\ref{prop_linear_spaces_and_Chow_variety}\ref{item_families_of_linear_spaces_determine_morphism_to_Chow} there is a morphism 
     from $\Hom^{lin,n}$ to the Chow variety of $X$.
   By construction this morphism is surjective onto $\ccR'$. 
   In particular, $\ccR'$ is constructible and $\ccR'$ contains a Zariski open subset which is dense in closure of $\ccR'$.
   Let $\ccR$ be the normalisation of the closure of $\ccR'$,
      $\pi\colon \ccU_{\ccR} \to \ccR$ the Chow pullback of the universal family and suppose that $\xi$ is the composed evaluation map.
   Clearly $\pi$ is equidimensional and  $\xi^*L$ is $\pi$-ample as above, with general fibre $\PP^k$, and $\xi^*L|_{\PP^k} \simeq \ccO_{\PP^k}(1)$.
   Again by \cite[Prop.~4.10]{araujo_druel_codim1_del_Pezzo_foliations} all fibres are $\PP^k$ and $\pi$ is a family of linear spaces.
   Thus all points in the closure of $\ccR'$ represent $k$-dimensional linear subspaces.
\end{prf}

Therefore if $X$ is projective and $L$ is ample, then we may briefly say that 
 the reduced scheme (that is, a union of algebraic varieties) parametrising lines or linear subspaces is also projective.
In particular, the property of being connected by a line is Zariski closed,
   that is, the set of pairs $(x,y) \in X \times X$,
   such that $x$ and $y$ are connected by a line is Zariski closed in $X \times X$.

\subsection{Lines on projective varieties}
\label{sect_varieties_with_many_lines}

We have the standard consequence of the Mori's Bend and Break Lemma:

\begin{lemma}[{\cite[Cor.~II.5.5.2]{kollar_book_rational_curves}}]\label{lem_bend_and_break_1_point}
   Suppose that $X$ is a projective variety and $L$ is an ample line bundle on $X$.
   Consider a positive dimensional family of distinct lines, that is a map $\ccU_{\ccR} \to \ccR$ as above, 
      with $\ccR$ irreducible, $\dim \ccR>0$, and distinct fibres $\PP^1 \subset \ccU_{\ccR}$ mapped to distinct lines in $(X,L)$.
   Then the family may have at most $1$ common point.
\end{lemma}

The following is a generalisation, which is also standard, but hard to reference explicitly.
In fact, the proof is very similar to the proof of Mori's bend and break theorem.
It treats the case when the family has a common point and claims that the lines in the family have distinct tangent directions.
Similar statements appear in \cite[Thms~1.3 and 1.4]{kebekus_kovacs_are_rational_curves_determined_by_tangents}, 
   but here we do not assume that the point $x$ is general, or that the tangent direction is general.
   
\begin{lemma}\label{lem_bend_and_break_tangent_directions}
   Suppose that $X$ is a projective variety and $L$ is an ample line bundle on $X$.
   Consider a positive dimensional family of distinct lines on $(X,L)$ 
     parametrised by a proper and irreducible $\ccR$ with $\dim \ccR>0$ 
     and a point $x \in X$ common to all the lines in the family.
   If all the lines have at worst nodal singularities at $x$,
     then there are at most finitely many lines in this family that have a fixed tangent direction at $x$.
\end{lemma}

\begin{prf}
  Suppose on the contrary that there is a positive dimensional set of lines with a common tangent direction.
  Without loss of generality, we may assume that $\ccR$ is a projective curve and all the lines in the family have a common tangent direction.
  Normalising $\ccR$ we may assume the curve is smooth.
  By Lemma~\ref{lem_can_take_U_normal} we may also assume the total space $\ccU_{\ccR}$ for the family is normal.
  By Proposition~\ref{prop_linear_spaces_and_Chow_variety_for_projective}\ref{item_families_of_linear_spaces_are_bundles}
    the total space $\ccU_{\ccR}$ is a projectivisation of a vector bundle $\ccE$ of rank $2$ over $\ccR$, $\ccU_{\ccR} \simeq \PP(\ccE)$.
  The evaluation map $\xi \colon \PP(\ccE) \to X$ contracts the image of a section $\sigma\colon \ccR\to\PP(\ccE)$ to $x\in X$.
  
  To make a further simplification, we may replace $X$ with the image of the evaluation map $\xi$. 
  Hence $X$ is a surface.
  We may also replace $X$ by its normalisation, that is, assume $X$ is normal.
  
  In this setting, let $\tilde{X}$ be the blow up of $X$ at $x$ with an exceptional divisor which is the projectivisation of the tangent cone of $X$ at $x$.
  In particular, the strict transform of each line in the family passes through a single point $\tilde{x}$ by our assumption about the tangent directions.
  Let $D\subset \tilde{X}$ be an irreducible component of the exceptional divisor of the blow up containing $\tilde{x}$.

  Let $\tilde{\xi}\colon \tilde{\ccU}\to \tilde{X} $ be the minimal resolution of the rational map $\PP(\ccE) \to X \dashrightarrow \tilde X$:
  \[
     \xymatrix{  \tilde{\ccU} \ar[r]^{\tilde{\xi}} \ar[d] & \tilde{X} \supset D\ar[d]^{\text{blow up of $x$}} \\
                 \PP(\ccE) \ar[r]^{\xi} \ar[d]^{\pi} & X \ni x  \\
                 \ccR \ar@/^/[u]^{\sigma}}
  \]
  Thus $\tilde{\ccU}$ is a blow up of the smooth ruled surface $\PP(\ccE)$ in several points $\fromto{s_1}{s_k}$. 
  Consider the strict transform of the section $\sigma(\ccR)$. 
  It is contracted by the morphism to $\tilde{X}$ to the point $\tilde{x}$.
  We claim that none of the blown up point $s_i$ is on the section $\sigma(\ccR)$.
  Otherwise, let $C$ be the fibre of $\pi$ through such $s_i \in \sigma(\ccR)$. 
  The strict transform $\tilde{C}$ of $C$ in $\tilde{\ccU}$ is disjoint from the strict transform of $\sigma(\ccR)$,
     and by minimality of the resolution, $\tilde{\xi}(\tilde{C})$ is also disjoint from $\tilde{x}$, a contradiction.
  Thus the rational map $\PP(\ccE) \dashrightarrow \tilde X$ is regular near $\sigma(\ccR)$.
  
  It follows that the preimage of $D$ under $\PP(\ccE) \dashrightarrow \tilde X$ has an irreducible component $Q$ which intersects properly $\sigma(\ccR)$, and is contractible in $\PP(\ccE)$.
  Elementary intersection theory shows that such divisor cannot exist on the ruled surface $\PP(\ccE)$. 
  That is, suppose that $Q$ is numerically equivalent to $\alpha \sigma(\ccR) + \beta F$, where $F$ is a class of a fibre of $\pi$.
  Then:
  \begin{align*}
       \sigma(\ccR)^2 &<0,&  \sigma(\ccR)\cdot F &=1,& F^2& =0, \\ 
       \alpha(\alpha \sigma(\ccR)^2 + 2 \beta) = Q^2 &<0, & \alpha \sigma(\ccR)^2 + \beta = Q \cdot \sigma(\ccR)&>0, & \alpha = Q \cdot F&\ge 0,
  \end{align*}
  which has no solutions for $\alpha$ and $\beta$.
\end{prf}

The following conclusion can be informally interpreted as follows.
In the projective setting, if $\ccR$ is a family of lines through a fixed point, then for any sufficiently small locally closed analytic space $\ccR'$ containing a general point of $\ccR$ the locus of lines from $\ccR'$ 
  is generically an analytic locally closed submanifold. In particular, it makes sense to discuss the tangent space of such locus, as we do in Section~\ref{sect_first_jump}.

\begin{cor}\label{cor_locus_of_lines_has_a_submanifold}
   Suppose that $X$ is a projective variety with an ample line bundle $L$.
   Consider a family of lines $\ccU \to \ccR$  such that $\ccR$ is irreducible and projective and the evaluation map 
      $\xi\colon\ccU \to X$ is generically finite to one 
   Fix a general point $c\in \ccR$ and any locally closed analytic submanifold $\ccR' \subset \ccR$ such that $c\in \ccR'$.
   Let $C \subset X$ denote the curve corresponding to $c$.
   Denote by $R' \subset X$ the locus of $\ccR'$, that is, $R'$ is the union of lines in $X$ that are represented by points of $\ccR'$.
   Then $R'$ contains locally closed submanifold $T$ of $X$ 
      that contains a Zariski open subset of $C$, 
      and such that $\dim T=\dim \ccR' +1$. 
   In particular, the Zariski closures of $R'$ and $T$ coincide.
\end{cor}

Note that indeed,
    if $\ccR$ is a family of pairwise different 
    lines passing through a fixed point $x\in X$ then the evaluation map 
      $\xi\colon\ccU \to X$ is generically finite to one by Lemma~\ref{lem_bend_and_break_1_point}, 
      thus the assumption of the corollary is satisfied.

\begin{prf}
   For simplicity, assume $\xi\colon \ccU \to X$ is dominant (hence surjective) by replacing $X$ with the image of $\xi$.
   The evaluation map $\xi\colon \ccU \to X$ is generically {\'e}tale by \cite[Lem.~III.10.5]{hartshorne}.
   Denote by $\ccU^0 \subset \ccU$ the open subset where $\xi$ is {\'e}tale. 
   Since $\xi$ is projective and generically finite to one, 
      the image $Z:=\xi(\ccU \setminus \ccU^0) \subset X$ is closed and strictly contained in $X$ (which is irreducible).
   Let $\ccU^1\subset \ccU^0$ be the preimage of the complement of this strict subset, 
   \[
     \ccU^1 =  \xi^{-1} \left(X \setminus Z\right).
   \]
   Thus $\ccU^1$ is a Zariski open dense subset of $\ccU$, while $\xi(\ccU^1) = X \setminus Z$ is Zariski open and dense in $X$. 
   The restricted map $\xi^1 \colon \ccU^1 \to \xi(\ccU^1)$ is:
   \begin{itemize}
    \item {\'e}tale, since it is a restriction of the {\'e}tale map $\xi|_{\ccU^0}$ to an open subset, and 
    \item projective, since it is the base change of the projective map $\xi$ under the open embedding $X \setminus Z \subset X$ \cite[Tag~01WF]{stacks_project},
    \item finite by \cite[Cor.~12.89]{gortz_wedhorn_algebraic_geometry_I},
    \item a finite topological covering.
   \end{itemize}

   Let $\PP^1_{c} \subset \ccU$ denote the preimage of $\set{c}$ under $\ccU\to \ccR$, so that $C=\xi(\PP^1_c)$. 
   Since $c$ is general in $\ccR$, $\ccU^1 \cap \PP^1_{c}$ is a non-empty (hence dense) Zariski open subset of $\PP^1_c$,
   and thus also  $\ccU^1 \cap \ccU|_{\ccR'}$ is (Euclidean) open and dense in $\ccU|_{\ccR'}$.
   Consider 
   \[
     (\xi^1)^{-1}\left(\xi^1(\ccU^1 \cap \PP^1_{c})\right) = (\ccU^1 \cap \PP^1_{c}) \sqcup Y,
   \]
   where $Y\subset \ccU^1$ is a (Zariski) closed submanifold.
   By the topological $T_4$-axiom for $\ccU^1$ considered with the Euclidean topology,
   there exists a Euclidean open subset $\ccU^2 \subset \ccU^1$ which separates $\ccU^1 \cap \PP^1_{c}$ from $Y$, or more precisely, 
     $\ccU^1 \cap \PP^1_{c} \subset \ccU^2$ and the Euclidean closure $\overline{\ccU^2}$ is disjoint from $Y$.
     
   Denote $T^2:=\xi^1(\ccU^2 \cap \ccU|_{\ccR'})$, 
     which locally is a finite union of locally closed analytic submanifolds of dimension equal to $\dim \ccU|_{\ccR'} = \dim \ccR +1$, 
     hence $T^2$ is a locally analytic subvariety of $X$ of the same dimension. 
   Moreover, $T^2 \subset R'$, $T_2$ contains a Zariski open dense subset of $C$, 
      and $T_2$ is smooth at a general point of $C$ by the choice of $\ccU^2$.
   Hence the smooth locus $T$ of $T^2$ satisfies the required properties.
\end{prf}

\section{Distributions}\label{sect_distributions}

In this subsection we summarise some basic material about distributions.
We essentially follow the convention of \cite[\S2]{hwang_mok_birationality},
  where a distribution is an equivalence class of subbundles defined on some open subsets.

\begin{defin}\label{def_distribution}
   Suppose that $E$ is a vector bundle on an analytic space or an algebraic variety $X$.
   Consider a pair $(G_U, U)$, where $U \subset X$ is an open dense subset and $G_U$ is a vector subbundle of $E|_{U}$
   (in particular, $G_U$ is itself  a vector bundle on $U$).
   Two pairs $(G_U, U)$ and $(G'_{U'}, U')$ are equivalent if and only if $G_U|_{U\cap U'} =G'_{U'}|_{U\cap U'}$ as subbundles of $E|_{U\cap U'}$.
   A \emph{distribution} $G$ in $E$ (also denoted $G \subset E$) is an equivalence class of such pairs.
\end{defin}
An alternative definition of a distribution is as a subsheaf.
\begin{lemma}
   Let $X$, $E$ be as in the definition above. 
   Given a distribution $G \subset E$, we can define its sheaf (also denoted $G$) of sections as the subsheaf of sections of $E$ whose images over $U$ are contained in $G_U$, whenever $(G_U, U)$ is a pair in the equivalence class $G$.
   The quotient sheaf $E/G$ is torsion free.
   Conversely, a subsheaf $G \subset E$, such that $E/G$ is torsion free, uniquely determines a distribution in $E$, whose sheaf of sections is $G$.
\end{lemma}
The proof is elementary, and since we are not going to use the sheaf definition, we skip the proof.

Let $X$, $E$, and $G$ be as in Definition~\ref{def_distribution}. 
The \emph{rank} of $G$ is the rank of $G_U$ as the vector bundle on $U$.
Note that rank does not depend on the choice of $U$ as any two open dense subsets must intersect.
We let $U(G)$ be the maximal open (dense) subset of $X$ such that 
a pair $(G_{U(G)}, U(G))$ is in the equivalence class $G$.
Note that if $X$ is normal, then $X \setminus U(G)$ is always of codimension at least $2$ in $X$.
In particular,  if $X$ is a smooth curve, then $G$ is always a vector subbundle of $E$.

The word ``distribution'' is usually associated with a vector subbundle of a tangent bundle.
In this article we need a slightly more general situation: we will also consider distributions (for example) in
the restriction of the tangent bundle $TX|_Y$ or in the normal bundle $N_{Y \subset X}$, where $Y$ is a closed subvariety of $X$.

All sorts of natural operations can be applied to distributions.
If $G_1$ and $G_2$ are distributions in $E$, then $G_1 + G_2$ and $G_1 \cap G_2$ also are.
If $f \colon Z \to X$ is a morphism, and the image of $f$ intersects $U(G)$,
  then $f^* G$ is a distribution in $f^*E$, etc.
We write $G_1 \subset G_2$, if the inclusion holds over $U(G_1) \cap U(G_2)$.

One of the situations we will often consider is when $Y \subset X$ is a subvariety or an analytic subspace,
   and the distribution is $G\subset TX|_{Y}$.
For example, the tangent bundle $T Y \subset TX|_{Y}$ is a distribution
  (note that $Y$  need not be smooth, $U(T Y)$ is the smooth locus of $Y$).
We will say that a subvariety, or an analytic subset $Z \subset Y$ is \emph{$G$-integral},
   if $Z$ intersects $U(G)$ and $T Z \subset G|_{Z}$ as distributions in $TX|_Z$.
We say $Z$ is a \emph{leaf} of $G$, if $Z$ is $G$-integral and $\dim Z = \rk G$.

\subsection{Singularities and rank \texorpdfstring{$1$}{1} distributions}
In Section~\ref{sect_divisors_of_non_standard_lines} we will work with a rank $1$ distribution on a projective variety $B$, which will be restricted to various subvarieties in $B$ (or their normalisations).
In that setting $B$ is singular, but nevertheless, we will slightly abuse (or abbreviate) our notation and write $G \subset TB$ to mean the following:
\begin{notation}\label{notation_distribution_on_singular}
   Whenever $B \subset X$ is a singular subvariety of a smooth variety $X$ and $G$ is a distribution $G \subset TX|_B$ such that $G \subset TB$ as a distributions in $TX|_B$, 
   we simply write $G \subset TB$, even though $TB$ is not defined everywhere on $B$.
\end{notation}
   Note that the choice of smooth $X$ containing $B$ is irrelevant to the questions about the distribution $G\subset TB$.
   In particular, we do not need to mention $X$ explicitly.

If $G\subset TB$ is a rank $1$ distribution, 
  then  $G$ is a \emph{foliation}, that is analytically locally there exist leaves of $G$ through general points of $B$. 
More precisely, such leaves exist  through any (smooth) point of $U(G)$, and they are locally unique 
  (that is, if $\Delta_1\subset B$ and $\Delta_2 \subset B$ are two leaves containing a common point $x \in \Delta_1\cap\Delta_2$,
  such that $x\in U(G)$ is a smooth point, then $\Delta_1 \cap \Delta_2$ is also a leaf and it is open in both $\Delta_1$ and $\Delta_2$).
In this subsection we review some elementary lemmas about such distributions.

The following lemma (in a more general setting) is found in \cite[Thm~3.8]{deserti_cerveau_foliations_and_groups} or in \cite[\S4.1]{araujo_druel_on_Fano_foliations}.
\begin{lemma}\label{lem_distribution_on_Pk}
   Suppose that $G \subset T \PP^k$ is a rank $1$ distribution on a projective space~$\PP^k$,
      which after the restriction to a general line $\PP^1 \subset \PP^k$
      is
   \[
      \ccO_{\PP^1}(1) \subset T \PP^k|_{\PP^1} \simeq
      \ccO(1^{k-1}, 2).
   \]
   Then there exists a point $y \in \PP^k$ such that all the lines through $y$ are tangent to $G$,
   that is, the leaves of $G$ are those lines. In particular, the leaves of $G$ are algebraic.
\end{lemma}

Another classical fact about rank $1$ distributions  is the following rectification property.
Briefly, it claims that it makes sense to consider ``the union of leaves of $G$ in $B$ through general points of a curve''.

\begin{lemma}\label{lem_union_of_generic_leaves}
   Subject to Notation~\ref{notation_distribution_on_singular}, 
      let $B$ be a projective variety, and $C \subset B$ is a closed curve intersecting the smooth locus of $B$.
   Consider a rank one distribution $G \subset TB$, which is generically transversal to $C$.
   Then there exists a Euclidean open subset $U \subset B$,
      which intersects $C$ in a dense Zariski open subset $C\cap U\subset C$,
      and a closed analytic submanifold $\Gamma\subset U$, 
     such that $\dim \Gamma =2$, $C \cap U \subset \Gamma$, and $G|_{\Gamma} \subset T \Gamma$.
\end{lemma}
This lemma is similar to the discussion in \cite[\S1.2]{bogomolov_mcquillan_rational_curves_on_foliatede_varieties}.
It also resembles 
\cite[Prop.~4.5]{brunella_uniformisation_of_foliations_by_curves}
or
\cite[\S4.1]{campana_paun_foliations_with_positive_slopes}.
However, in all these approaches one can extract appropriate surface near (in the Euclidean topology) each point of $C$.
The missing part we need is gluing these locally defined surfaces to one surface containing a Zariski dense subset of $C$.
Since we do not have compactness neither of smooth locus of $B$ or of regular locus of $G$ (even after restricting to $C$), 
   it is a little subtle to obtain finiteness of the local covers.
For instance, in \cite[\S4.1]{campana_paun_foliations_with_positive_slopes} 
  the authors construct an universal leaf $\Lambda$, which is an analytic submanifold of the product $B\times B$ containing an open dense subset of the diagonal and the surface $\Gamma$ we are looking for should be just 
  the image under second projection of the preimage of $C$ under the first projection $\Lambda \to B$. 
  However, since the relevant subsets are not compact, 
     we are not aware of any statement that would guarantee 
     that this image is an analytic subspace.

The argument is elementary and uses standard topological methods, but tedious, so we skip the details, leaving only the following sketch.

\begin{prf}[ of Lemma~\ref{lem_union_of_generic_leaves} (sketch)]
   Let $Y \subset B$ be a Zariski open dense subset contained in the smooth locus of $B$,
      where $G$ is defined and such that $C$ is smooth and transversal to $G$ at points of $C \cap Y$.
   The construction of $U$ is such that $C\cap U = C\cap Y$, and $U$ is a union of small neighbourhoods of points of $C\cap Y$. On each such small neighbourhood we can find the appropriate analytic surface using Rectification Theorem \cite[Thm~1.18]{ilyashenko_yakovenko_lectures_on_analytic_differential_equations}.
   Then use the \emph{paracompactness} of $Y$ and its subsets \cite{dieudonne_paracompact_spaces} to observe that it is enough 
   to consider only a locally finite family of such small open subsets.
   Conclude that the small analytic surfaces glue together to a well defined surface $\Gamma$ with the desired properties.
\end{prf}

\begin{defin}\label{def_union_of_generic_leaves}
   We say that $\Gamma$ as in Lemma~\ref{lem_union_of_generic_leaves} is 
   a \emph{surface obtained as the union of leaves of $G$ through general points of~$C$}. 
\end{defin}

\subsection{Manifolds with global corank 1 distributions}\label{sect_corank_1_distributions}

\begin{defin}\label{def_global_corank_1_distribution}
Suppose that $X$ is a complex manifold and $F \subset T X$ is a distribution,
  such that $U(F) = X$, that is, a distribution defined on the whole of $X$.
Suppose that $\rk F = \dim X -1$ and let $\theta: T X \to T X /F =:L$
  be the quotient map, so that the following is a short exact sequence of vector bundles on $X$:
\[
  0\to F \to T X \stackrel{\theta}{\to} L \to 0.
\]
In this situation we say that $(X,F)$ is a \emph{manifold with a global corank $1$ distribution}.
\end{defin}
We stress the word \emph{global}. In the definition above we assume $F \subset TX$ is a vector subbundle.

\begin{observation}\label{obs_corank_1_distribution}
   Suppose that $(X, F)$ is a manifold with a global corank $1$ distribution,
      and $Y \subset X$ is an analytic subset.
   Let $Y_0$ be the smooth locus of $Y$ and
      consider a distribution $G \subset T Y_0$,
      which is defined as $G:= T Y_0 \cap F$.
   Then either $Y$ is $F$-integral,
      or there exists an open dense subset $Y'\subset Y$,
      such that $(Y', G|_{Y'})$ is a manifold with a global corank $1$ distribution.
\end{observation}
The above observation captures the motivation for our treatment of manifolds with global corank $1$ distributions.
That is, even though our primary interest is in contact manifolds (see \S\ref{sect_prelim_contact_manifolds}),
  in our arguments we will prove claims about subvarieties of contact manifolds,
  and they have the property of being (generically)  manifolds with corank $1$ distributions.
As a side result, some of our intermediate results apply to a more general situation, than just contact manifolds.

\begin{prop}[{\cite[Prop C.1(i) and (iv)]{jabu_dr}}]\label{properties_of_distribution}
  Let $(X,F)$ be a manifold with a global corank $1$ distribution and $L$ as in Definition~\ref{def_global_corank_1_distribution}.
  \begin{enumerate}
    \item
      \label{item_dtheta_well_defined}
     the locally defined derivative $\ud \theta$ determines a well defined homomorphism of vector bundles:
      \[
        \ud \theta: \Wedge{2} F \to L.
      \]
      Specifically, $\theta \colon TX\to L$ locally is a $1$-form on $X$ (after choosing a local trivialisation of $L$). 
         The locally defined derivative of $\theta$ is a $2$-form, which depends on the choice of the  trivialisation of $L$. 
         Its restriction to $F$ does not depend on this choice, and thus it glues to a globally defined map $\Wedge{2} F \to L$.
    \item
      \label{item_Y_tangent_to_F}
      If $\Delta \subset X $ is $F$-integral, $\Delta_0$ is the smooth locus of $\Delta$,
         then $\ud \theta|_{\Delta_0} \equiv 0$.
      In particular:
      \[
        \dim \Delta \le \rk F - \half \min_{x \in \Delta} \left( \rk \ud\theta_x \right).
      \]
  \end{enumerate}
\end{prop}
\noprf

\begin{notation}\label{def_perp}
\renewcommand{\theenumi}{(\alph{enumi})}
Suppose that $(X,F)$ is a manifold with a global corank $1$ distribution, as in Definition~\ref{def_global_corank_1_distribution}. 
\begin{enumerate}
 \item \label{item_notation_dtheta}
    Item \ref{item_dtheta_well_defined} 
    of Proposition~\ref{properties_of_distribution} 
    implies that for every $x\in X$ there is a skew-symmetric bilinear form 
  $\ud \theta_x \colon F_x \times F_x \to L_x \simeq \CC$.
Its rank $\rk \ud\theta_x$ (that is, the rank of the induced linear map $F_x \to F_x^* \otimes L_x$) is an even integer due to skew-symmetry.
Moreover, the rank is semicontinuous as a function of $x$, 
  that is, the subsets $\set{x \in X | \rk \ud \theta_x \le 2k}$ are closed analytic subspaces for any integer $k$.
In particular, for every subvariety $Y \subset X$, there is an open dense subset of $Y$, where $\rk \ud \theta_x$ is constant.
\item \label{item_notation_perp} 
   Suppose that $Y \subset X$ is a subvariety.
   Consider the vector bundle $TX|_Y$ on $Y$ and another distribution $G$ in this vector bundle $TX|_{Y}$
   (we stress, that $G$ is not assumed to be a distribution in $TX$, that is, it needs not to be defined over a general point of $X$;
    it is only defined on an open subset of $Y$).
In this situation by $G^{\perp_F}$ we denote the distribution in $TX|_{Y}$:
\begin{equation}\label{equ_define_perp}
   G^{\perp_F} := (G\cap F|_{Y})^{\perp_{\ud \theta}}\subset F|_{Y}.
\end{equation}
Here $\perp_{\ud \theta}$ denotes the perpendicular subspace with respect to the skew form defined in \ref{item_notation_dtheta}.
This distribution is defined on an open dense subset of $Y$ 
  where the rank of   $G\cap F|_{Y}$ is constant and where
  $\rk \ud \theta$ is constant.
\item 
  For any $x\in X$, the \emph{degeneracy locus} of $F$ at $x$ 
    is $(F_x)^{\perp_{\ud \theta}} \subset F_x$.
  The \emph{degeneracy subdistriution} of $F$ is the distribtion in 
    $TX$ determined by the degeneracy loci at general points of $X$.
  Equivalently, the degeneracy subdistribution of $F$ is $F^{\perp_F}$.
  For consistency, we say that the degeneracy subdistribution of $TX$ is $TX$ (this would be a degenate case, when $\theta=0$,
    and arises from restricting $F$ to smaller subvarieties which are integral).
  \end{enumerate}
\end{notation}
\renewcommand{\theenumi}{(\roman{enumi})}

\begin{defin}
In the situation of Definition~\ref{def_global_corank_1_distribution}, 
consider the open dense subset $X_0 \subset X$, where $\rk \ud\theta_x$ is constant and equal to $2r$ for $x \in X_0$.
Suppose that $\Delta \subset X_0$  is an analytic submanifold.
We say $\Delta$ is \emph{maximally $F$-integral} if $\Delta$ is $F$-integral and $\dim \Delta = \dim X - r -1$.
\end{defin}

This is an analogue of a Legendrian subvariety in contact manifold, but here $F$ need not to be a contact structure on $X$
  (that is, $\ud \theta_x$ is not necessarily non-degenerate).
  
\begin{lemma}\label{lem_maximally_integral}
   Suppose that $(X,F)$ is a complex manifold with a global corank $1$ distribution, $L$, $\theta$ 
   as in Definition~\ref{def_global_corank_1_distribution},
     and $X_0\subset X$ is the open dense subset where $\rk\ud\theta_x$ is constant.
   Pick $x \in X_0$ and let $v \in T_x X$.
   Then $v$ is in the degeneracy locus of $F$,
      if and only if $v \in T_x \Delta$ for every maximally $F$-integral analytic submanifold $\Delta \subset X_0$ containing $x$.
\end{lemma}

\begin{proof}
   We have $(F_x)^{\perp_{\ud \theta}} = \set{v\in F_x\colon \forall w \in F_x, \ud \theta_x (v,w) =0}$.
   Pick a maximally $F$-integrable analytic submanifold $\Delta\subset X_0$ with $x \in \Delta$.
   We first prove that $(F_x)^{\perp_{\ud \theta}} \subset T_x \Delta$.
   Suppose otherwise that $v\notin T_{x}\Delta$ and $v\in (F_x)^{\perp_{\ud \theta}}$.
   Let $W = T_x \Delta + v$ be the vector space of dimension $(\dim \Delta +1) = \dim X - r$.
   It is an isotropic subspace, because $v$ is perpendicular to all vectors in $F_x$, 
     and $T_x \Delta$ is perpendicular to itself.
   But the maximal possible dimension for an isotropic subspace is $\dim F_x -r$, a contradiction.

   We now prove the other implication.
   The problem is analytically local around $x$, so we can assume $X = X_0$ is an analytically open subset of $\CC^{\dim X}$,
   $x=0$, $L\simeq \ccO_X$ is a trivial line bundle, and $\theta\in H^0(T^*X \otimes L) \simeq H^0(T^*X)$
      is in the Darboux normal form $\theta = \ud x_{0} - \sum_{i=1}^{r} x_i \ud x_{r+i}$.
   Clearly $F_x = \set{\ud x_{0} = 0}$ and $(F_x)^{\perp_{\ud \theta}} = \set{\ud x_{0} = \dotsb = \ud x_{2r}=0}$.
   Let $\Delta_1 =  \set{x_{0}= x_1 = \dotsb x_r =0}$ and $\Delta_2 =  \set{x_{0}= x_{r+1} = \dotsb x_{2r} =0}$.
   These are maximally $F$-integral submanifolds of $X_0$ containing $x$, so
   \begin{multline*}
     v \in T_0 \Delta_1 \cap T_0 \Delta_2 =\\
        \set{\ud x_{0} = \ud x_{1} = \dotsb = \ud x_{r}=0} \cap
        \set{\ud x_{0} = \ud x_{r+1} = \dotsb = \ud x_{2r}=0} = \\ (F_x)^{\perp_{\ud \theta}}.
   \end{multline*}
\end{proof}
All the manifolds with global corank $1$ distributions come with
  a natural polarisation in the sense of the opening paragraph of Section~\ref{sect_preliminaries}.
  Namely, $(X,L = TX/F)$ is a polarised manifold.
  Here $L$ is a line bundle, since it is a quotient of the vector bundle of rank $\dim X$ by its subbundle of rank $\dim X -1$.
\begin{lemma}\label{lem_lines_are_integral}
   Let $X$ be a manifold with a global corank $1$ distribution $F$ and $L, \theta$ as in Definition~\ref{def_global_corank_1_distribution}. 
   If $\Gamma \subset X$ is a linear subspace of $(X,L)$, then $\Gamma$ is $F$-integral.
   In particular, lines are always $F$-integral.
\end{lemma}
\noprf

\section{Contact manifolds}\label{sect_prelim_contact_manifolds}

Let $(X,F)$ be a manifold with a global corank $1$ distribution, with the short exact sequence
   $0\to F \to T X \stackrel{\theta}{\to} L \to 0$.
In particular, $L$ is a line bundle that is going to be used to measure the degrees.
Recall the map $\ud \theta \colon \Wedge{2}F \to L$ from Proposition~\ref{properties_of_distribution}.
As defined in Section~\ref{sect_introduction}, if $\ud \theta $ is nowhere degenerate then 
   $(X, F)$ is a contact manifold.
In particular, $\ud \theta$ makes $F_x$ into a symplectic vector space for each $x\in X$ and $\dim X$ is odd.

If $X$ is a contact manifold of dimension $2n+1$,
   then $-K_X$ is a divisor linearly equivalent to the Cartier divisor of the line bundle $L^{\otimes (n+1)}$.
Our main interest is when $X$ is projective, in fact Fano, which is therefore equivalent to $L$ being ample.
However, some of the statements are true in a more general setting.

Maximally integral submanifolds (or subvarieties) of contact manifolds are called \emph{Legendrian}.
That is:

\begin{defin}\label{def_Legendrian}
  A subvariety, or an analytic subspace (or a reduced subscheme) $Y \subset X$ is \emph{Legendrian}, if it is of (pure) dimension $n$ and
     $T Y \subset F|_{Y}$ (as distributions in $T X|_{Y}$).
\end{defin}

\subsection{Splitting types on special lines}\label{sect_splitting_types}

In this subsection we suppose that $X$ is a complex contact manifold,
   with a contact distribution $F \subset TX$ and the quotient line bundle $L  = TX/F$.
In particular, $X$ does not need to be projective, or compact.
   
Suppose that $f \colon \PP^1 \to X$ is a holomorphic map such that $f^*L \simeq \ccO(1)$, 
   that is $f$ is a parametrisation of a line.
We consider $f^* TX$. By \cite[Prop.~2.8]{4authors} the splitting type of this vector bundle is
\[
   f^*TX 
   = \ccO(\fromto{a_1}{a_n},\fromto{-b_1}{-b_{n+1}})
\]
  with $a_i >0$ and $b_j\ge 0$.
We also have $c_1(f^*TX) = n+1$, so $\sum a_i - \sum b_j = n+1$.
Since the differential gives a non-zero morphism $T\PP^1\simeq \ccO(2) \to f^*TX$,
  we must have at least one $a_i\ge 2$.
In particular, in the case of a standard line, as defined in \eqref{equ_standard_splitting_type},
  we have  $a_1=2$, $a_2=\dotsb=a_n=1$, $b_1 = \dotsb = b_{n+1} =0$.

Further consider $f^* F$. Since $F^* \simeq F \otimes L$, we also have:
\[
  f^*F \simeq
\ccO(\fromto{c_1}{c_n},\fromto{1-c_1}{1-c_n})
\]
for some integers $c_i > 0$. In particular, there are exactly $n$ strictly positive entries in this splitting.

\begin{lemma}\label{lem_positive_parts_of_splitting_are_in_F}
  Suppose that $(X,F)$ is a complex contact manifold with line bundle $L=TX/F$ and $f\colon \PP^1 \to X$ is a parametrisation of a line on $(X,L)$. 
  Then the short exact sequence
  \[
    0 \to f^*F \to f^* TX \stackrel{f^* \theta}{\to} \ccO(1) \to 0
  \]
  does not split, and $c_i =a_i$.
\end{lemma}
\begin{prf}
  Suppose on the contrary, that the exact sequence splits. Then one of the integers $a_i$ is equal to $1$, say $a_n=1$.
  Further 
  \[
    f^*F = \ccO(\fromto{a_1}{a_{n-1}}, \fromto{-b_1}{-b_{n+1}}),
  \]
    a contradiction, since there are only $n-1$ strictly positive entries in this splitting.

  Consider the restriction $f^* \theta \colon\ccO(\fromto{a_1}{a_n}) \to \ccO(1)$.
  Since $a_i \ge 1$ and the sequence does not split, this restriction must be identically zero.
  So the positive part comes from $f^*F$ and $c_i = a_i$.
\end{prf}

Now suppose that there is only one $b= b_n \ge 0$ and the remaining $b_1=\dotsb=b_{n-1}=0$.
Equivalently, there is at most one negative term in the splitting of $f^*TX$.
Then one of the integers $a_i$, say $a_n$, must satisfy  $1-a_n \le -b$.
So suppose that $a_n = b+c$ for some $c > 0$.
We must have   $\left(\sum_{i=1}^{n-1} a_i\right) + c = n+1$, and we conclude:
\begin{lemma}\label{lem_splitting_types}
   If there is at most only one negative term in the splitting of $f^*TX$, then either
   \begin{align*}
      f^*TX & =  \ccO(-b, 0^{n}, 1^{n-2}, 2, b+1) \text{ and}\\
       f^*F & =  \ccO(-b, -1, 0^{n-2}, 1^{n-2}, 2, b+1),\\ \text{ or}\\
      f^*TX & =  \ccO(-b, 0^{n}, 1^{n-1}, b+2) \text{ and}\\
       f^*F & =  \ccO(-b-1, 0^{n-1}, 1^{n-1}, b+2)
   \end{align*}
   for some $b\ge 0$.
   Note that in the second case the image of $f$ must have cuspidal singularities (unless $b=0$),
      that is, the curve is not an immersed curve.
\end{lemma}

The lemma applies to a situation where there is a set of lines filling in a divisor.
\begin{lemma}\label{lem_number_of_non_negative_terms}
   Assume $(X,L)$ is a polarised complex manifold and $X$ is either compact of Fujiki class $\ccC$ or quasi-projective.
   Suppose that $\ccB$ is an irreducible (analytic or algebraic) variety parametrising lines on $X$,
      and let $B \subset X$ be the locus swept by those lines.
   Assume $c\in \ccB$ is a general line from $\ccB$, and let $C \subset B$ be the corresponding curve,
      with a birational parametrisation $f\colon \PP^1 \to C$.
   Then $f^*TX$ has at most $\codim (B\subset X)$ negative terms in its splitting,
      and the distribution $f^*T B \subset f^*T X$ is contained in $(f^*T X)^{\ge 0}$.
\end{lemma}

\begin{prf}
   Since $C$ is a rational curve, we can replace $\ccB$ with $\Hom_{\ccB}$, a subvariety of $\Hom(\PP^1 \to X)$,
     consisting of the morphisms that are birational onto their images, and whose images are the curves in $\ccB$.
   The locus swept by $\Hom_{\ccB}$, that is, the union of images of all morphisms from $\Hom_{\ccB}$, is equal to $B$.
   In this situation $f \in \Hom_{\ccB}$ is a general point, in particular, it is a smooth point of $\Hom_{\ccB}$,
      even though $f$ might be a singular or non-reduced point of $\Hom(\PP^1 \to X)$.
   Similarly, if $p \in \PP^1$ is a general point, then $f(p)$ is a general point in $B$.

   Thus all the tangent directions in $T_f \Hom_{\ccB} \subset T_f \Hom(\PP^1 \to X)$  
      can be realised as curves in $\Hom_{\ccB}$, that is, as deformations of $C$, which (in particular) are contained in $B$.
   We have $T_f \Hom(\PP^1 \to X) = H^0(f^*TX)$ 
      and the differential of the evaluation map $\Hom(\PP^1 \to X) \times \PP^1 \to X$ 
      at $(f,p)$ is the evaluation of sections $H^0(f^*TX) \to f^*(TX)_{p} = T_{f(p)} X$ 
      \cite[Prop.~II.3.4]{kollar_book_rational_curves}.
   The deformations obtained from $\Hom_{\ccB}$ sweep out $B$,
      so the image of the evaluation contains $T_{f(p)} B$,
      which is only possible if the number of non-negative terms in the splitting of $f^*TX$ is at least
      $T_{f(p)} B = \dim B$.
   Equivalently, the number of negative terms is at most $\codim (B\subset X)$.
\end{prf}

\begin{cor}\label{cor_non_standard_lines_do_not_cover_X}
   Assume $(X,F)$ is a complex contact manifold and $X$ is either compact of Fujiki class $\ccC$ or quasi-projective.
   Suppose that $\ccB$ is an irreducible analytic or algebraic subvariety of the Barlet space or of the Chow variety,
      whose general member is a non-standard line.
   Then:
   \begin{enumerate}
    \item    the locus $B$ swept by $\ccB$ does not cover $X$, and
    \item    if, in addition, $\codim (B\subset X) =1$,
               then the splittings of $f^*TX$ and $f^*F$ 
               are as in Lemma~\ref{lem_splitting_types}, with $b<0$,
               where $f$ is the normalisation of the general line from $\ccB$.
   \end{enumerate}
\end{cor}
\noprf

For future reference we note the following lemma:
\begin{lemma}\label{lem_perp_of_TB}
   Assume as above that $(X,F)$ is a complex contact manifold and $X$ is compact of Fujiki class $\ccC$ or quasi-projective.
   Suppose that $\ccB$ is an irreducible analytic or algebraic subvariety of the Barlet space or of the Chow variety,
      whose general member is a non-standard line and $\codim (B\subset X) =1$, where $B \subset X$ is the locus swept by $\ccB$.
   Let $G \subset TX|_{B}$ be the rank $1$ distribution defined as 
      \[
         G = TB^{\perp_F} = (TB \cap F|_{Y})^{\perp_{\ud \theta}}.
      \]
   Consider a general line in $\ccB$ and its parametrisation $f \colon \PP^1\to X$.
   Then $f^*G \subset f^*TX$ is equal to $\ccO(b+1) \subset \ccO(-b, 0^{n}, 1^{n-2}, 2, b+1)$ 
       or $\ccO(b+2) \subset \ccO(-b, 0^{n}, 1^{n-1}, b+2)$.
   In the first case, if in addition $b=1$, we may choose the splitting of $\ccO(2, 2)= \ccO(2,b+1) = \ccO(2)\oplus \ccO(b+1)$ 
       in such a way that $G$ is the second summand.
   In the latter case, $f^*G = T\PP^1$ as distributions in $f^*TX$, that is, the general line in $\ccB$ is tangent to $G$.
\end{lemma}
\begin{proof}
   Generically, we have $f^*TB = (f^*TX)^{\ge 0}$ and 
   \begin{align*}
       f^*TB \cap f^*F = (f^*TX)^{\ge 0} \cap f^*F & =  \ccO(-1, 0^{n-2}, 1^{n-2}, 2, b+1) \text{ or }\\
                                                   & =  \ccO(-1, 0^{n-2}, 1^{n-1}, b+2).
   \end{align*}
   (see Lemma~\ref{lem_splitting_types} and Corollary~\ref{cor_non_standard_lines_do_not_cover_X}).
   Thus $f^*TB \cap f^*F$ extends to a vector subbundle of $f^*TX$, and the degree of its perpendicular line bundle
   \[
      G = (f^*TB \cap f^*F)^{\perp_{\ud \theta}} \simeq \left(f^*F/ (f^*TB \cap f^*F)\right)^*\otimes f^*L 
   \]
   is $b+1$ in the first case or $b+2$ in the second case.
   The remaining statements are straightforward.
\end{proof}

\subsection{Parameter spaces for lines on contact Fano manifolds}
\label{sect_parameter_spaces_for_lines_on_contact}

Let $(X,F)$ be a contact Fano manifold of dimension $2n+1$.
A lot of attention aims to understand lines on $X$.
Let us underline, that we mean lines with respect to the polarisation $L=TX/F$, as in \S\ref{sect_lines_and_linear_subspaces}.
In particular, if $X = \PP^{2n+1}$, then $L \simeq \ccO_{\PP^{2n+1}}(2)$ and there are no lines on $X$.
In all the other projective cases, the lines exist and cover $X$:
  for $X = \PP(T^* M)$, the ordinary lines in the fibres are lines with respect to $L$;
  for $X$ Fano with $\Pic X = \ZZ[L]$, this is observed (for example) in \cite[\S2.3]{kebekus_lines1}.

By Proposition~\ref{prop_linear_spaces_and_Chow_variety_for_projective}\ref{item_linear_spaces_are_closed} the set of lines on $(X,L)$ is Zariski closed in the Chow variety.  
Let $\ccH$ be  an irreducible component of this set.
Thus each point in $c \in \ccH$ represents an irreducible rational curve $C\subset X$
   with normalisation $f\colon \PP^1 \to C$, and  $f^*L = \ccO_{\PP^1}(1)$.

If $Z$ is a scheme we denote by $\reduced{Z}$ the reduced subscheme of $Z$.   
As in \S\ref{sect_lines_and_linear_subspaces}, from the definition of the Chow variety, $\ccH$ comes with the following diagram:
\[
  \xymatrix{ & \ccU_{\ccH} \ar[dr]^{\phi} \ar[dl]_{\pi}&\\ \ccH&& X}
\]
where $\ccU_{\ccH}$ is the \emph{universal family},
  that is the subvariety of $\ccH \times X$, such that the projection $\pi: \ccU_{\ccH} \to \ccH$
  is equidimensional and the set-theoretic fibre $\reduced{\inv{\pi}(c)}$ is $\set{c} \times C$, $C$ is the curve corresponding to $c$.
The map $\phi: \ccU_{\ccH} \to X$ is a projection on the second factor.

For $x \in X$, we let $\ccH_x$ be the \emph{scheme of lines through $x$} defined as $\reduced{\pi(\inv{\phi}(x))}$.
Here, although $\ccH_x$ has a reduced scheme structure, we avoid the word ``variety'', since we cannot claim that $\ccH_x$ is always irreducible.
We also let $H_x \subset X$ be the \emph{union of lines through $x$}, that is $\reduced{\phi(\inv{\pi}(\ccH_x))}$.

\begin{rem}
   Kebekus in his presentation of \cite{kebekus_lines1}, \cite{kebekus_lines2}
     assumes in addition that $\ccH$ \emph{dominates} $X$,
     that is, there exits a line from $\ccH$ that passes through a general point of $X$.
   However, this assumption is redundant, and below we briefly explain why.
\end{rem}

\begin{prop}[{\cite[Prop.~4.1]{kebekus_lines1}}]\label{prop_Hx_Legendrian}
  Suppose that $X$ is a contact Fano manifold of dimension $2n+1$, not isomorphic to $\PP^{2n+1}$.
  Let $x \in X$ be any point, 
     and $\ccH$ an irreducible component of the subset of Chow variety parametrising lines.
      Suppose that $\ccH_x$ is the scheme of lines through $x$ and $H_x$ is the union of lines though $x$ as above.
  Then $H_x \subset X$ is  Legendrian, in particular, it is of pure dimension $n$.
  Furthermore,
   $\ccH_x$ is of pure dimension $n-1$.
\end{prop}
\begin{prf}
   By the standard dimension estimates, coming from the deformation theory and Riemann-Roch for $\PP^1$,
     we have $\dim \ccH_x \ge n-1$ (compare to \cite[Equation~(2.1)]{kebekus_lines1}).
   Moreover,  $\dim H_x = \dim \ccH_x +1$ by Mori's bend and break (Lemma~\ref{lem_bend_and_break_1_point}).
   Thus it remains to prove that $H_x$ is $F$-integral.

  Let $y\in H_x$ be a general (smooth) point of an irreducible component of $\ccH_x$,
     let $c \in \ccH_x$ represent a curve $C\subset X$
     that contains $x$ and $y$ and is smooth at $y$. 
  Let $f:\PP^1\to C$ be the normalisation map with $p\in \PP^1$ mapped to $x$.
  By \cite[Prop.~II.3.4]{kollar_book_rational_curves} 
  \[
    T_y H_x \subset T_y C + H^0(\PP^1, f^*(T_X) \otimes \ccO_{\PP^1}(-p))|_y.
  \]
  By Lemma~\ref{lem_lines_are_integral} it is enough to show that the second summand is contained in $F_y$.
  This follows from Lemma~\ref{lem_positive_parts_of_splitting_are_in_F}, 
    since all the sections of $f^*(T_X) \otimes \ccO_{\PP^1}(-1)$
    must come from sections of $f^*(F) \otimes \ccO_{\PP^1}(-1)$.
 \end{prf}

\begin{cor}\label{cor_dim_ccH_eq_3n-1}
  With $X$ and $\ccH$ as in Setting~\ref{sett_contact_Fano} we have
     $\dim \ccH = 3n-1$, and $\ccH$ dominates $X$.
\end{cor}

\begin{prf}
   As in \cite[(2.1)]{kebekus_lines1}, $\dim \ccH \ge 3n-1$.
   On the other hand
   \[
       \dim \ccH = \dim \ccH_x + \dim H -1,
   \]
     where $H \subset X$ is the locus swept by $\ccH$ (that is, the union of all the lines from $\ccH$),
     and $x\in H$ is a general point.
   Since $\dim \ccH_x = n-1$ and $\dim H \le \dim X = 2n+1$,
     we must have $\dim \ccH = 3n-1$ and $H = X$.
\end{prf}

Similarly, if $\ccB, \ccS \subset \ccH$ are some subfamilies of lines,
  then we define $\ccU_{\ccB}, \ccU_{\ccS} \subset \ccU_{\ccH}$,
  $\ccB_x, \ccS_x\subset \ccH_x$, $B_x, S_x\subset H_x$ in an analogous way.
Typically, $\ccB$ will be the family of non-standard lines,
  and $\ccS$ will be the family of singular lines, or rather they will be some irreducible components of these families.

\section{Singular lines}\label{sect_sing_lines}

Throughout this section we will assume the following setting:
\begin{setting}\label{sett_for_singular_lines}
Suppose that $X$ is a complex manifold, and $L$ is a line bundle on $X$.
In addition, we impose our favourite assumptions on $X$: either $X$ is compact of Fujiki class $\ccC$ or $X$ is quasi-projective. 
Consider the Barlet space of $X$ or, respectively, the Chow variety of $X$.
Inside the cycle space consider an irreducible component $\ccS$ of the set of singular lines on $(X,L)$
  (see~\S\ref{sect_barlet_and_chow}).
Let $\ccU_{\ccS} \subset \ccS \times X$ be the universal family for $\ccS$ and let $S \subset X$
  be the image of the projection $\ccU_{\ccS} \to X$.
That is, $S$ is the locus swept by lines from $\ccS$.
\end{setting}

Note that we do not claim that $\ccS$ or $S$ are closed in the cycle space or in $X$.
   As explained in \S\ref{sect_barlet_and_chow}, $\ccS$ is locally closed with closure $\overline{\ccS}$, which is an analytic variety, and such that the boundary $\overline{\ccS} \setminus \ccS$ 
   is also locally closed. 
Moreover, $S$ is constructible by Chevalley-Remmert Theorem \cite[Thm~on~p291 and Rem.~on~p293]{lojasiewicz_intro_to_complex_analytic_geometry}.

\subsection{Kebekus results on singular lines in the projective case}\label{sect_kebekus_singular_line}

Deformations of singular rational curves are studied by Kebekus in \cite{kebekus_families_of_singular_rational_curves}
  (in general setting) and in \cite[\S3]{kebekus_lines1} (in the setting of contact Fano manifolds).
The summary of these results (restricted to the setting of lines) is presented in the following proposition.

\begin{prop}\label{prop_Kebekus_singular_lines}
   Let $(Y,L_{Y})$ be a polarised projective variety (not necessarily smooth, or normal).
   Suppose that $L_{Y}$ is ample and $y\in Y$ is a general point.
   Then the set of singular lines through $y$ is at most finite, and all these lines are smooth at $y$.
   If $Y$ is in addition a contact Fano manifold with the contact distribution $F\subset TY$
     such that $L_Y\simeq TY/F$,
     then all the lines through $y$ are smooth.
\end{prop}
For a proof see \cite[Thm~3.3(2)]{kebekus_families_of_singular_rational_curves} and \cite[Prop.~3.3]{kebekus_lines1}.
We will also generalise the latter result and proof in Proposition~\ref{prop_singular_lines}.
We will need the following combination of these two statements,
   that states that dimension of the variety parametrising singular contact lines on a contact Fano manifold is at most $2n-1$.

\begin{prop}\label{prop_dimension_of_the_space_of_singular_lines}
 In Settings~\ref{sett_contact_Pic_generated_by_L} and \ref{sett_for_singular_lines} combined we have $\dim \ccS \le 2n-1$.
\end{prop}

\begin{prf}
  By Proposition~\ref{prop_Kebekus_singular_lines} applied to $Y=X$ and $L_Y=L$ (contact case),
     $S \neq X$, thus $\dim S \le \dim X - 1 =2n$.
  Further, by Proposition~\ref{prop_Kebekus_singular_lines} applied to $(Y,L_Y)=(S, L|_S)$ (general case),
     there are finitely many singular lines through a general point of $S$.
  Thus $\dim \ccS = \dim S -1 \le 2n-1$ as claimed.
\end{prf}

\subsection{Singular lines as morphisms}

An integral singular rational curve can be always dominated by an
  integral singular plane cubic, that is, by a rational curve with a
  single node or cusp.
For that reason, (in Setting~\ref{sett_for_singular_lines}) let $Q$ be a singular plane cubic,
  and let $\Hom^{lin}(Q \to X)$ be the normalisation of the space of morphisms $f\colon Q\to X$
  such that the degree of the line bundle $f^* L$ is $1$.
Note that such morphism is automatically birational onto its image.

\begin{lemma}\label{lem_Hom_dominating_S}
In Setting~\ref{sett_for_singular_lines}, there exists a singular plane cubic $Q$ and 
 an irreducible component $\Hom_{\ccS}$ of the analytic space 
  \mbox{$\Hom^{lin}(Q \to X)$},
  which dominates $\ccS$, that is, we have the commutative diagram:
\[
  \xymatrix{
 & \Hom_{\ccS} \times Q \ar[ddr]^{\ev} \ar[dl]_{pr_1}\ar[d]_{im_Q} &   \\
 \Hom_{\ccS}\ar[d]^{im} & \ccU_{\ccS} \ar[dr]^{\phi} \ar[dl]_{\pi}& \\
   \ccS&& X}
\]
and the map $im \colon \Hom_{\ccS} \to \ccS$ is dominant.
Here $\ev(f, q)  = f(q)$ is the evaluation map, $pr_1$ is the projection map,
   $im(f)$ is the image curve, considered as a point in $\ccS$,
   and $im_Q(f, q) = (im(f), f(q))$.
  Moreover, fix a smooth point $q \in Q$.
  Then the map $im_Q|_{\Hom_{\ccS} \times \set{q}}\colon \Hom_{\ccS} \times \set{q} \to \ccU_\ccS$ is dominant too.
\end{lemma}

\begin{prf}
  Pick a general point $c \in \ccS$. 
  If $c$ represents a curve which has at least one node, pick $Q$ to be the nodal plane cubic.
  Otherwise, pick $Q$ to be the cuspidal plane cubic.
  Let $\PP^1 \to Q$ be the normalisation map.
  The map $\Hom^{lin}(Q \to X) \times \PP^1 \to\Hom^{lin}(Q \to X)$ together with a map $\xi$, which is the composition of the normalisation 
    and $\ev$, makes a family of lines.
  Thus there is a map from $\Hom^{lin}(Q \to X)$ to the cycle space by Proposition~\ref{prop_linear_spaces_and_Chow_variety}.
  All singular lines with a node or a cusp (depending on the singularity of $Q$) will be in the image of this map.
  In particular, by the generality of our choice of $c$, the component $\ccS$ is dominated by an irreducible component of $\Hom^{lin}(Q \to X)$, showing the first claim.
  
  To see that $im_Q|_{\Hom_{\ccS} \times \set{q}}$ is dominant, we note that the automorphism group of $Q$ acts transitively on the smooth points of $Q$.
  Compare also with \cite[Prop.~2.8]{kebekus_families_of_singular_rational_curves}.
\end{prf}

\subsection{Singular lines and distribution}\label{sect_sing_lines_and_distributions}
For the rest of this section we will assume in addition to Setting~\ref{sett_for_singular_lines} that 
  $0\to F \to T X \stackrel{\theta}{\to} L \to 0$
  is a short exact sequence of vector bundles as in \S\ref{sect_corank_1_distributions},
  so that $X$ is a manifold with a global corank $1$ distribution 
  (and the line bundle $L$ from Setting~\ref{sett_for_singular_lines} coincides with the quotient $TX/F$).
Later, we will also assume that $F$ is either generically or globally a contact distribution.

The locus $S$ swept by singular lines $\ccS$ contains an open dense subset $S^0$, that is a locally closed analytic subspace of $X$.
Define the distribution $G \subset T S^0$  by $G:=T S^0 \cap F|_{S^0}$ 
   (here, if $S^0$ is not smooth, $T S^0$ is a distribution in $T X|_{S^0}$, as discussed in Notation~\ref{notation_distribution_on_singular}).
By Observation~\ref{obs_corank_1_distribution},
  either $S^0$ is $F$-integral, or there exists a further open dense subset $S' \subset S^0$,
  such that $(S', G|_{S'})$ is a manifold with a global corank $1$ distribution.

\begin{prop}
  \label{prop_singular_lines}
  In Setting~\ref{sett_for_singular_lines}, assume in addition that $X$ has 
  a global corank $1$ distribution $F$ such that $L \simeq TX/F$.
  Suppose that $c \in \ccS$ is a general point corresponding to a singular line $C\subset X$.
  Then $C\cap S_0$ is generically tangent to the degeneracy subdistribution of $G|_{S_0}$.
\end{prop}
The proof of this proposition follows quite strictly the lines of the proof of \cite[Prop.~3.3]{kebekus_lines1},
  however our statement is stronger.

\begin{proof}
  Since $C\subset S$, the claim is clear 
    if $S^0$ is $F$-integral.
  Thus, using Observation~\ref{obs_corank_1_distribution},
    we may suppose that $S'$ is a manifold with a global corank~$1$
    distribution, where $S'$ is open and dense in $S$.

Pick a singular plane cubic $Q$ and an irreducible component $\Hom_{\ccS}$ as in Lemma~\ref{lem_Hom_dominating_S},
    so that the map $\ev_q:\Hom_{\ccS} \to S$, $\ev_q(f) = f(q)$ is dominant for any smooth point $q$ of $Q$ (compare to the discussion at the end of \S\ref{sect_barlet_and_chow}).

  For a general morphism $f\in \Hom_{\ccS}$ the tangent map of $\ev_q$ has the maximal
     rank at $f$, that is, $\rk T_f \ev_q= \dim S$.
  The set of pairs $(f,q)$ for which the rank is maximal is open in $\Hom_{\ccS} \times Q$.
  By \cite[II.6.10.2, II.6.11.4 and Ex.~II.6.7]{hartshorne}
    the smooth points of $Q$ are in 1:1-correspondence with line bundles
    of degree one.
  So fix a general $f\in \Hom_{\ccS}$, and a general point $q\in Q$,
    such that $\ccO_Q(q)\ncong f^*(L)$ and $\rk T_f \ev_q= \dim S$, see Lemma~\ref{lem_Hom_dominating_S}.
  Note that also $\ccO_Q(q')\ncong f'^*(L)$ and $\rk T_{f'} \ev_{q'}= \dim S$ for all $q'$ in a small neighbourhood of $q \in Q$ and $f'$ in a small neighbourhood of $f\in \Hom_{\ccS}$.  

  Let $s=f(q)$ and $C = f(Q)$. Note that $C$ and $s$ are general.
Suppose to the contrary of the claim of the proposition that $T_s C$ is not tangent to the degeneracy locus of $G$ at $s$. 
      By Lemma~\ref{lem_maximally_integral} applied to
      the manifold with global corank $1$ distribution $(S', G|_{S'})$,
      there exists $\Delta \subset S'$, a maximally $G$-integral analytic submanifold of $S'$,
      which is transversal to $C$ at $s$.
  Since $\ev_q$ has the maximal rank, we can find a section $\Gamma \subset \Hom_{\ccS}$
  over $\Delta$, that is, a submanifold $\Gamma$ such that
  $\ev_q|_{\Gamma}:\Gamma\to \Delta$ is biholomorphic near $f \mapsto s$.
  By the construction,  $\ev (\Gamma \times Q)$ contains a small analytic submanifold $\Delta'$ of $S'$
    of dimension $\dim \Delta +1$,
    that contains $\Delta$ and $C = f(Q)$.
  In particular, $\Delta'$ cannot be $G$-integral.

  So there exists $(f', q') \in \Gamma \times Q$ in a small neighbourhood of $(f,q)$,
    such that $T_{f'(q')} \Delta' \not \subset G_{f'(q')}$,
    so $\theta(T_{f'(q')} \Delta') \ncong 0$.
  But
  \begin{align*}
     T_{f'(q')} \Delta' &= T_{(f', q')} \ev (T_{f'} \Gamma + T_{q'} Q) = T_{f'} \ev_{q'} (T_{f'} \Gamma) + T_{f'(q')} f'(Q) \\
     & \subset  \set{\sigma(q') \mid \sigma \in H^0((f')^* T X)} + T_{f'(q')} f'(Q).
  \end{align*}
  Since $\theta ( T_{f'(q')} f'(Q))  =0$ by Lemma~\ref{lem_lines_are_integral},
    there must exist a section $\sigma \in H^0((f')^* T X)$,
    such that $(f')^*(\theta)(\sigma)(q') \ne 0$.
  But $(f')^*(\theta)(\sigma) \in H^0((f')^*L)$ and
  \[
    \sigma(q) \in T_{f'(q)} \Gamma \subset G_{f'(q)}
  \]
  So $(f')^*(\theta)(\sigma) \ne 0$ and it vanishes at $q$, a contradiction with our choice of $q$.
  So $T_s C$ is tangent to the degeneracy locus of  $G$ at $s$ as claimed.
\end{proof}

Suppose that $X$ is a generically contact manifold as defined in Section~\ref{sect_introduction}.
Translating it into the language introduced in 
Section~\ref{sect_corank_1_distributions}, 
$X$ is a manifold with a global corank $1$ distribution 
and the rank $\rk \ud \theta_x$ 
from Notation~\ref{def_perp}\ref{item_notation_dtheta} 
is equal to $\dim X-1$ (maximal possible value) for a general point $x$.
We obtain the following corollaries which generalise Proposition~\ref{prop_Kebekus_singular_lines} 
  to the situation, where either $X$ is a projective generically contact manifold,
  or $X$ is a contact manifold, which is not necessarily projective.
In the first situation, $X$ could be for instance a birational modification 
  of a projective contact manifold.
In the second situation, $X$ can be a quasi-projective contact manifold 
  (see \cite{hwang_manivel_quasi_complete_contact})
  or a compact contact manifold of Fujiki class $\ccC$ 
  (see \cite{frantzen_peternell}, \cite{peternell_jabu_contact_3_folds}).

\begin{cor}\label{cor_singular_lines_on_generically_contact}
  In Setting~\ref{sett_for_singular_lines}, 
  suppose in addition that $(X,F)$ is a generically contact manifold of dimension $2n+1$ such that $TX/F\simeq L$.
  Then $S \ne X$, that is, the singular contact lines do not cover $X$.
  If in addition $X$ is projective and $L$ is ample, then $\dim \ccS \le 2n-1 = \dim X -2$. 
\end{cor}

\begin{prf}
  Suppose that $\ccS$ is non-empty,  $c\in \ccS$ is a general singular line,
     and $s \in C$ is a general point on this line.
  Then by Proposition~\ref{prop_singular_lines},
    one has $T_s S \cap F_s \subset T_s C^{\perp} \subsetneqq F_s$.
  Thus $T_s S \subsetneqq T_s X$ and $S \ne X$.

  If $X$ is projective and $L$ is ample,
     then by Proposition~\ref{prop_Kebekus_singular_lines} (general case applied to $S$) 
     we get $\dim \ccS \le \dim S -1 \le \dim X -2$.
\end{prf}

\begin{cor}\label{cor_singular_lines_on_non_projective_contact}
  In Setting~\ref{sett_for_singular_lines} 
     suppose that $(X,F)$ is a (globally) contact manifold 
     such that $TX/F\simeq L$.
  If $\dim S = \dim X-1= 2n$,
    then $\dim \ccS = \dim X-2 = 2n-1$.
\end{cor}

\begin{proof}
  Let $S_0 \subset S$ be the smooth locus of $S$.
  Since $S$ has codimension $1$ in $X$,
    hence $G:=(T S)^{\perp F}$ is a rank one distribution in $T S_0$,
  and by Proposition~\ref{prop_singular_lines}, the singular lines in $\ccS$
    must be the leaves of $G$.
  In particular, there is a unique such line through a general point of $S$, and we must have
    $\dim \ccS = \dim S -1 = \dim X-2$ as claimed.
\end{proof}

\section{Divisors of non-standard lines}\label{sect_divisors_of_non_standard_lines}

%

Throughout this section we stick to the following setting.
\begin{setting}\label{setting_divisor_of_non-standard_lines}
   In addition to Setting~\ref{sett_contact_Fano}, we assume $\ccB\subset \ccH$ is a closed irreducible subset of codimension $1$ containing only non-standard lines., 
     $B\subset X$ is its locus. For $x\in X$ we denote by $H_x$ the locus of lines from $\ccH_x$. 
   If in addition $x\in B$, then $\ccB_x$ is the set of lines in $\ccB$ that pass through a point $x\in B$, $B_x$  is the locus of $\ccB_x$. 
   (See \S\ref{sect_parameter_spaces_for_lines_on_contact}.)
   If $c$, $c'$ are points in $\ccB$, they represent non-standard lines in $X$, which we denote by $C$ and $C'$ respectively. 
\end{setting}

By Corollary~\ref{cor_non_standard_lines_do_not_cover_X} we must have $B \ne X$.
The lemma below explains that the claim of \cite[Prop.~3.2]{kebekus_lines2} is equivalent to the claim that $\ccB$ as above does not exist.
(Recall from \S\ref{sec_history_lines} that the proof of \cite[Prop.~3.2]{kebekus_lines2} has a gap 
  which invalidates the claim of \cite[Thm~1.1(2)]{kebekus_lines2} that $\ccH_x$ is irreducible for a general point $x\in X$.)
\begin{lemma}\label{lem_ccB_x_is_a_union_of_irreducible_components}
   In Setting~\ref{setting_divisor_of_non-standard_lines}, 
   for all $x \in B$ the set $\ccB_x$ is a union of irreducible components of $\ccH_x$.
   In particular, $\dim \ccB_x = n-1$, and analogously, $B_x$ is a union of irreducible components of $H_x$ and $\dim B_x = n$.
   Moreover, $B$ is a divisor, that is, $\dim B =2n$.
\end{lemma}
\begin{prf}
   $\ccH_x$ is of pure dimension $n-1$, and $H_x$ is of pure dimension $n$  by Proposition~\ref{prop_Hx_Legendrian}.
   Thus 
   \[
    n-1 = \dim \ccH_x \ge \dim \ccB_x \ge  \dim \ccB + 1 - \dim B \ge n-1
   \]
   and thus the pure dimension of $\ccB_x$ is $n-1$ and $\ccB_x$ is a union of irreducible components of $\ccH_x$.
   Moreover, $\dim B = 2n$.
   Since each component of $H_x$ is a locus swept by a component of $\ccH_x$, it follows that $B_x$ is also a union of irreducible components of $H_x$.
\end{prf}

To prove Theorem~\ref{thm_dimension_of_the_space_of_special_lines} we must show that $B$ is covered by linear subspaces of dimension $n$.
We jump up and down to get this result.
By construction, we start knowing that $B$ is covered by lines, that is linear subspaces of dimension $1$.
First we jump up and claim that $B$ is covered by $\PP^k$'s for some $k\ge 2$.
Then we fall down and argue that $k=2$. 
Finally, we jump again and show that there is a lot of those planes, enough to fill in $\PP^n$.

\subsection{The first jump}\label{sect_first_jump}
In this subsection we show that the locus $B$ of an irreducible component of special lines of dimension $3n-2$ 
   is covered by linear subspaces of dimension $k$ for some $k\ge 2$.

The locus $B$ comes with the rank $1$ distribution $G:=(T B)^{\perp_F}  \subset T X|_{B}$ as in Notation~\ref{def_perp} and \eqref{equ_define_perp}.
We follow the abbreviation as in Notation~\ref{notation_distribution_on_singular}.
Below we compare the distribution $G$ to another distribution $\tilde G_{x} \subset T \ccB_x$ 
  for a general point $x\in B$, 
  which has the property that (in informal words) moving a curve along $\tilde G_{x}$ is the same as moving general point{s} of the curve along $G$.
This arises from comparing the locally closed submanifolds constructed in Corollary~\ref{cor_locus_of_lines_has_a_submanifold} and from Lemma~\ref{lem_union_of_generic_leaves}. 
In order to construct $\tilde{G}$, we use the \emph{Hilbert to Chow morphism}, and in fact we construct $\tilde{G}$ on a subset of the Hilbert scheme.

\begin{prop}[Hilbert to Chow morphism]\label{prop_Hilb_to_Chow}
    In Setting~\ref{setting_divisor_of_non-standard_lines}, 
      there exists an irreducible and reduced component $\Hilb$ of the Hilbert scheme of $X$ and a projective birational morphism
      $\varpi\colon\Hilb\to \ccH$ 
      which takes a point of $\Hilb$ representing a subscheme of $X$ to its underlying cycle in $\ccH$.
    Moreover, $\varpi$ is an isomorphism on the set of smooth curves.
    More precisely, if $\ccS\subset \ccH$ is the subset of singular lines, 
      then the restriction of $\varpi$ to $\varpi^{-1}(\ccH \setminus \ccS) \to \ccH \setminus \ccS$ is an isomorphism. 
\end{prop}
\begin{prf}
Pick a general point $c_{gen} \in \ccH$ in the irreducible component of lines.
This point represents a smooth and standard (in particular, free) rational curve $C_{gen} \subset X$, 
   which at the same time is a subscheme of $X$, hence represents a smooth point of the Hilbert scheme of $X$.
The tangent space to the Hilbert scheme at this point is the space of sections of the normal bundle $H^0(N_{C \subset X})$,
   see \cite[Thm.~I.2.8(1)]{kollar_book_rational_curves}.
By \eqref{equ_standard_splitting_type} the dimension of this tangent space is $3n-1$.

Let $\Hilb$ be the (unique) irreducible and reduced component of the Hilbert scheme that contains the point representing $C_{gen}$.
The Hilbert to Chow morphism is constructed in \cite[Cor.~3.3]{mangusson_Douady_Barlet_morphism} and 
    maps $\varpi\colon\Hilb \to \ccH$.
Note that this result of Mangusson is stronger than \cite[Thm~I.6.3(1)]{kollar_book_rational_curves}, 
    where the Hilbert to Chow morphism is constructed only from 
    the seminormalisation $\Hilb^{sn} \to \ccH$.
Since $C_{gen}$ is reduced and irreducible and it represents a smooth point of $\Hilb$ 
    (so the seminormalisation is an isomorphism near this point),
    by \cite[Thm~I.6.3(2)]{kollar_book_rational_curves} the Hilbert to Chow morphism is injective near this point. 
By dimension count $\dim \Hilb = \dim \ccH = 3n-1$ the map is also dominant.
Since both varieties $\ccH$ and $\Hilb$ are projective, the map is surjective.
Since a general curve $C_{gen}$ is represented by a single point in both $\Hilb$ and $\ccH$,
   it follows that the map is birational.

To complete the proof we must show that $\varpi$ is an isomorphism on the set of smooth curves.
To see this we produce the inverse map $\ccH \setminus \ccS \to \Hilb$.
The map comes from the universal property of the Hilbert scheme, as the universal family 
   $\ccU_{\ccH\setminus \ccS} \subset X\times (\ccH\setminus\ccS)$ is flat over $\ccH\setminus\ccS$.
Indeed, $L$ is ample on $X$, and the Hilbert polynomial $\chi(L|_{C}^d) = \chi(\ccO_{\PP^1}(d))$ is constant on all the fibres 
  (independent of the smooth curve $C \subset X$ represented by a point in $\ccH\setminus\ccS$),
  so by \cite[Thm~III.9.9]{hartshorne} the family is flat and induces a map from $\ccH\setminus\ccS$ to the Hilbert scheme. 
It is straightforward to verify that the map is an inverse of the Hilbert to Chow map $\varpi|_{\varpi^{-1}(\ccH\setminus\ccS)}$.
\end{prf}

\begin{cor}\label{cor_construction_of_G_tilde}
   In Setting~\ref{setting_divisor_of_non-standard_lines}, pick a point $x \in B$.
   \begin{enumerate}
      \item \label{item_tangent_space_to_ccB_x}
            If $c\in \ccB_x$ is a point which represents a smooth curve $C\subset X$, then 
      \[
        T_{c} \ccB_x \subset H^0(N_{C\subset X} \otimes \gotm_x) \simeq H^0(N_{C\subset X}(-1)),
      \]
      where $N_{C \subset X}$ is the normal bundle of $C$ in $X$ and $\gotm_x$ is the ideal sheaf of the point $x$ in the curve $C$.
      \item \label{item_splitting_for_normal_bdle}
            Suppose that $c\in \ccB$ is a general point.
            Then $c$ represents a smooth curve $C\subset X$, which is generically transversal to the distribution $G=TB^{\perp_{F}}$, and $N_{C \subset X}(-1) \simeq \ccO_C(-b-1,(-1)^n, 0^{n-2}, b)$ 
               for some integer $b\ge 1$.
      \item \label{item_defining_G_tilde}
            With $c$, $C$, and $b$ as in \ref{item_splitting_for_normal_bdle}, 
               assume $x\in C$ is a general point.
            We have two distinguished subspaces of the $(n+b-1)$-dimensional vector space $H^0(N_{C\subset X}(-1))$:
               the $(n-1)$-dimensional tangent space $T_c \ccB_x$ (which depends on $x$) and the $(b+1)$-dimensional space $H^0(\ccO_{C}(b))$ coming from the unique (up to rescaling) embedding $\ccO_C(b) \hookrightarrow N_{C\subset X}(-1)$.
            Then their intersection $\tilde{G}_{x, c} := T_c \ccB_x \cap H^0(\ccO_C(b))$ is $1$-dimensional (that is, of expected dimension).
  \end{enumerate}
\end{cor}

\begin{prf}
     Since $c$ represents a smooth curve $C\subset X$, 
        the Hilbert to Chow morphism $\varpi$ is an isomorphism near the preimage of $c$ by Proposition~\ref{prop_Hilb_to_Chow}.
     In particular, $\varpi$ determines a natural isomorphism $T_{c} \ccH = T_{\varpi^{-1}(c)}\Hilb$, 
        and the latter is equal to $H^0(N_{C \subset X})$ by \cite[Thm.~I.2.8(1)]{kollar_book_rational_curves}. 
     Thus $T_{c} \ccB_x \subset H^0(N_{C \subset X})$ and the infinitesimal deformations contained in $\ccB_x$ must 
        vanish at $x$. Therefore  $T_{c} \ccB_x \subset H^0(N_{C \subset X} \otimes \gotm_x)$ as claimed in \ref{item_tangent_space_to_ccB_x}.

     In the setting of \ref{item_splitting_for_normal_bdle}, note that since $c$ is a general element of $\ccB$, 
        thus it is smooth and generically transversal to $G$ by a dimension count:
         by Theorem~\ref{thm_classification_in_low_dimension} we have $n\ge 2$, thus $\dim \ccB =3n-2 > 2n-1$.
     Thus $C$ is smooth by Proposition~\ref{prop_dimension_of_the_space_of_singular_lines}
         and there is a family of dimension at most $2n-1$  of algebraic curves tangent to $G$.
     The splitting type follows from Lemma~\ref{lem_splitting_types}. The integer $b$ is positive by our choice of $\ccB$ 
     (it has non-standard lines only).
     
     Finally, in \ref{item_defining_G_tilde}, reverting the order of choices of $x$ and $c$,
        note that $x$ is a general point of $B$ and $c \in \ccB_x$ is a general (in particular smooth) point of an irreducible component of $\ccB_x$.
     Since the locus of this irreducible component of $\ccB_x$ is an irreducible component of $B_x$,
        which has dimension $n$ (Lemma~\ref{lem_ccB_x_is_a_union_of_irreducible_components}), for a general point $y\in C$ the evaluation of sections in $T_c \ccB_x\subset H^0(N_{C\subset X} \otimes \gotm_x)$ at $y$ must be (at least)
        an $(n-1)$-dimensional space ($1$ dimension tangent to $B_x$ at $y$ comes from the tangent space to $C$, 
        the remaining $n-1$ dimensions come from moving $c$).
     But if the intersection of $T_c \ccB_x$ and $H^0(\ccO_C(b))$ has dimension higher than $1$,
        then the evaluation of the sections from $T_c \ccB_x$ at any point $y \in C$ can have dimension at most $n-2$,
        a contradiction.
     Therefore the intersection $\tilde{G}_{x, c}$ is $1$-dimensional as claimed. 
\end{prf}

\begin{defin}
     In Setting~\ref{setting_divisor_of_non-standard_lines}, let $x\in B$ be a general point. 
     We define the rank $1$ distribution $\tilde{G}_x \subset T \ccB_x$ (with the abbreviation as in Notation~\ref{notation_distribution_on_singular}),
       to be $\tilde{G}_{x,c} \subset T_c \ccB_x$ for a general point $c\in \ccB_x$ in any of the irreducible components, 
       where $\tilde{G}_{x,c} $ is as in Corollary~\ref{cor_construction_of_G_tilde}\ref{item_defining_G_tilde}.
\end{defin}

\begin{cor}\label{cor_G_is_O_of_b_plus_1}
  In Setting~\ref{setting_divisor_of_non-standard_lines}, with $c\in \ccB$ general (and $C\subset X$ the corresponding curve), 
     the restrictions of the distributions $TB \subset TX|_{B}$ and $G \subset TX|_{B}$ 
     to the smooth rational curve $C$ are (more precisely, can be extended to) the following vector subbundles:
  \[
     TB|_{C} = \ccO_C(0^{n}, 1^{n-2}, 2, b+1), \qquad G|_{C} = \ccO_C(b+1),
  \]
  such that the image of $G|_{C}$ under the quotient $TX|_C\to N_{C\subset X}$ is also $\ccO_C(b+1)$.
\end{cor}
\begin{prf}
  This follows from Lemma~\ref{lem_perp_of_TB} together with the fact that $C$ is smooth 
    and $TC$ is generically transversal to $G|_{C}$, 
    see Corollary~\ref{cor_construction_of_G_tilde}\ref{item_splitting_for_normal_bdle}.
\end{prf}

Since $\rk G = 1$, the distribution $G$ must be integrable, so locally there exists a small leaf $\Delta_x$ 
  through a general point $x \in B$,
  that is $T \Delta_x = G|_{\Delta_x}$.
In the course of the proof, we will see that the Zariski closure of each leaf of $G$ is a line. 

We will define a subset $\Gamma_c \subset B$ (that depends on the choice of $c$), which is extremely important for our arguments.
Informally, $\Gamma_c = \bigcup_{x \in C} \Delta_x$ is the union of leaves through general points of $C$
(see Lemma~\ref{lem_union_of_generic_leaves} 
   and Definition~\ref{def_union_of_generic_leaves}).
The Zariski closure of $\Gamma_c$ is our candidate for a linear subspace of dimension $k\ge 2$, as in the claim of the first jump.
In informal words, $\Gamma_c$ is a surface obtained by perturbing the points of $C$ in the directions of $G$.
We will show that it is also obtained by perturbing the line $C$ in the direction of $\tilde{G}$.

\begin{defin}
  Suppose that $\Gamma_1$ and $\Gamma_2$ are two smooth connected locally analytic subsets of a fixed variety. 
  We say $\Gamma_1$ and $\Gamma_2$ \emph{generically agrees},
     if their intersection $\Gamma_1\cap \Gamma_2$  
     contains a nonempty subset that is Euclidean open  in both $\Gamma_1$ and $\Gamma_2$.
\end{defin}
So if $\Gamma_1$ and $\Gamma_2$ generically agree, we may think of them as if they are analytic continuations of one another. 
In particular, the Zariski closures of $\Gamma_1$ and $\Gamma_2$ are equal.

\begin{prop}\label{prop_gamma_and_tilde_gamma_generically_agree}
   In Setting~\ref{setting_divisor_of_non-standard_lines}, fix a general point $c\in \ccB$,
      which represents a smooth non-standard line $C\subset X$,
      and fix a general point $x\in C$.
   Let $\ccR_{x, c}\subset \ccB_x$ be a small leaf of $\tilde{G_x} \subset T\ccB_x$ 
      (again, see Notation~\ref{notation_distribution_on_singular})
      through $c \in \ccB_x$.  
   Consider the following smooth connected two-dimensional locally analytic subsets of $B$ that contain a Zariski open subset of points of $C$:
   \begin{itemize}
    \item One which we denote by $\Gamma_c$ that, in addition to the above conditions,
            satisfies $T \Gamma_C \supset G|_{\Gamma_C}$, and is constructed in Lemma~\ref{lem_union_of_generic_leaves} 
            (union of $G$-leaves through general points of $C$).
    \item One which we denote by $\tilde{\Gamma}_{x, c}$ that, in addition to the above conditions, 
            is Euclidean open in the locus of lines in $\ccR'=\ccR_{x, c} \subset \ccB_x$ 
            as constructed in Corollary~\ref{cor_locus_of_lines_has_a_submanifold}.
   \end{itemize}
   Then these subsets generically agree.
\end{prop}
\begin{prf}
   Let $y\in \tilde{\Gamma}_{x, c}$ be a general point 
      (that is, for any point $y$ in an unspecified open dense subset of $\tilde{\Gamma}_{x, c}$ the following arguments are going to work).
   Thus there is a line $C'$ corresponding to $c' \in \ccR_{x, c}$ such that $y \in C'$.
   Then $T_y \tilde{\Gamma}_{x, c}$ is equal to the evaluation at $y$ of the set of sections in $H^0(TX|_{C'})$ 
      that are mapped to $T_{c'} \ccR_{x,c} = \tilde{G}_{x,c'} \subset H^0(N_{C' \subset X})$
      under the quotient $H^0(TX|_{C'}) \to H^0(N_{C' \subset X})$.
   Thus by the characterisations of $\tilde{G}_x$ (Corollary~\ref{cor_construction_of_G_tilde}) and of $G$ (Corollary~\ref{cor_G_is_O_of_b_plus_1}),
      we have $T_y \tilde{\Gamma}_{x, c} = T_y C' + G_y$.
   In particular, $G|_{\tilde{\Gamma}_{x, c}} \subset T \tilde{\Gamma}_{x, c}$, and thus the (sufficiently small) leaves of $G$ through (general) points of $\tilde{\Gamma}_{x, c}$ are contained in $\tilde{\Gamma}_{x, c}$.
   
   Note that $\ccR_{x,c}$ generically agree with  $\ccR_{x,c'}$, as they are both leaves of the same distribution $\tilde{G}_x$.
   In particular, by the generality of our choice of $c$, the above statement also holds for points of $C$:
   the leaves of $G$ through general points of $C$ are contained in $\tilde{\Gamma}_{x, c}$.
   
   To conclude, note that both subsets $\Gamma_c$ and $\tilde{\Gamma}_{x, c}$ contain a Zariski dense open subset of $C$, 
      in particular, their intersection is non-empty.
   By the construction of $\Gamma_c$ and the above considerations, they both contain the leaves of $G$ through the general points of $C$.
   Therefore they generically agree.
\end{prf}

Define $\overline{\Gamma_c}$ to be the Zariski closure of $\Gamma_c$.
Note that potentially, $\overline{\Gamma_c}$ is of a dimension higher than $2$, but we will observe this is not the case.

\begin{lemma}\label{lemma_any_points_on_Gamma_c_are_connected_by_a_line}
   In Setting~\ref{setting_divisor_of_non-standard_lines}, let $c\in \ccB$ be a general point.
   Then any two points $x,y \in \overline{\Gamma_c}$ are connected by a contact line from $\ccB$
      contained in $\overline{\Gamma_c}$.
\end{lemma}
\begin{proof}
   We will say that $conn(Z_1,Z_2)$ holds for subsets $Z_i \subset \overline{\Gamma_c}$,
      if for all $x\in Z_1$ and $y\in Z_2$,
      the points $x$ and $y$ are connected by a line from $\ccB$
      contained in $\overline{\Gamma_c}$.
   First, note that if $conn(Z_1, Z_2)$ holds, then $conn(\overline{Z_1}, \overline{Z_2})$ also holds.
   Indeed, the set of lines contained in $\overline{\Gamma_c}$ is Zariski closed 
      (see Proposition~\ref{prop_linear_spaces_and_Chow_variety_for_projective}\ref{item_linear_spaces_are_closed})
      and so the set of $y \in \overline{\Gamma_c}$ such that $conn(x, y)$ holds for a fixed $x$ is also Zariski closed.
   Swapping the roles of $x$ and $y$ we obtain the claim.

   Suppose that $x \in C$ is a general point and $y\in \Gamma_c$ is a sufficiently general point in a small neighbourhood of $C$.
   In the notation of Proposition~\ref{prop_gamma_and_tilde_gamma_generically_agree}, we have a sequence of ``generically agree'' relations, which we denote by $\sim$:
   \[
     \Gamma_c \sim \tilde{\Gamma}_{x,c} \sim \tilde{\Gamma}_{x,c'} \sim \Gamma_{c'} \sim \tilde{\Gamma}_{y,c'},
   \]
   where $c'\in \ccR_{x,c}$ represents a line that connects $x$ and $y\in \tilde{\Gamma}_{x,c}$.
   The first, third and fourth $\sim$ follow from Proposition~\ref{prop_gamma_and_tilde_gamma_generically_agree}.
   The second $\sim$ follows since $\ccR_{x,c} \sim \ccR_{x.c'}$.
   In particular, the Zariski closures of each of these five sets are equal.
   
   Let $z\in \Gamma_c$ be any point. 
   Thus $z\in \overline{\tilde{\Gamma}_{y,c'}}$, hence $conn(\Gamma_c, y)$ holds.
   By our choice of $y$, also $conn(\Gamma_c, \Gamma_c)$ holds.
   Taking the Zariski closure, we  obtain the desired $conn(\overline{\Gamma_c}, \overline{\Gamma_c})$.
\end{proof}

From Theorem~\ref{thm_many_lines} and Lemma~\ref{lemma_any_points_on_Gamma_c_are_connected_by_a_line} 
  we conclude the claim of the first jump:
\begin{cor}\label{cor_normalisation_of_Gamma_c_is_Pk}
   The normalisation of $\overline{\Gamma_c}$ is $\PP^k$.
\end{cor}
\noprf

\subsection{The fall}

In this subsection we prove:
\begin{prop}\label{prop_Fall}
   In Setting~\ref{setting_divisor_of_non-standard_lines} and as in Corollary~\ref{cor_normalisation_of_Gamma_c_is_Pk}
   we have:
   \[
     k= \dim \overline{\Gamma_c} =2.
   \]
\end{prop}

\begin{prf}
Let $\mu: \PP^k \to \overline{\Gamma_c}$ be the normalisation map,
      and consider the distribution $\mu^* G \subset T \PP^k$.
Consider $\tilde{C}\simeq \PP^1 \subset \PP^k$ to be the line that dominates $C$ via the normalisation map $\mu$.
Note that $\tilde{C}$ is mapped isomorphically onto $C$.
We have three distributions in $TX|_{\tilde{C}}$: $T \PP^k|_{\tilde{C}}$, $G|_{\tilde{C}}$, $T \tilde{C}$.
The first one contains the latter two and they are generically transversal to each other.
We claim that $G|_{\tilde{C}}$ as a distribution in $T \PP^k|_{\tilde{C}}$ is isomorphic to
   $\ccO_{\PP^1}(1) \subset \ccO_{\PP^1}(1^{k-1},2)$.
To see that divide out by $T \tilde{C} \simeq \ccO_{\PP^1}(2)$
   and recall that the image of $G|_{\tilde{C}}$ in $TX|_{\tilde{C}}=\ccO_{\tilde{C}}(-b, 0^{n}, 1^{n-2}, 2, b+1) $
   is the $\ccO_{\PP^1}(b+1)$ component (Corollary~\ref{cor_G_is_O_of_b_plus_1}).
Then, after a suitable choice of splitting $N_{\tilde{C} \subset \PP^k} \simeq \ccO_{\PP^1}(1^{k-2}) \oplus \ccO_{\PP^1}(1)$,
   the derivative map restricted to $\tilde{C}$:
\[
   N_{\tilde{C} \subset \PP^k} \simeq \ccO_{\PP^1}(1^{k-2}) \oplus \ccO_{\PP^1}(1) \to N_{C \subset X} \simeq \ccO_{\PP^1}(-b, 0^{n}, 1^{n-2}, b+1)
\]
is given by an embedding of $\ccO_{\PP^1}(1^{k-2})$ into $\ccO_{\PP^1}(1^{n-2})$,
and the remaining component $\ccO_{\PP^1}(1)$ is mapped non-trivially into $\ccO_{\PP^1}(b+1)$.
Therefore $G|_{\tilde{C}} \subset T \PP^k|_{\tilde{C}}$ must correspond exactly to this component $\ccO_{\PP^1}(1)$, 
  again after a suitable choice of the splitting
  $T \PP^k|_{\tilde{C}} = N_{\tilde{C} \subset \PP^k} \oplus T \tilde{C}$.

We have obtained a rank $1$ distribution $\mu^* G \subset T \PP^k$, such that its restriction to a general  line $\PP^1 \subset \PP^k$ is $\ccO_{\PP^1}(1) \subset T \PP^k|_{\PP^1}$.
Such a distribution can only be obtained as tangent to lines through some fixed distinguished point  $y \in \PP^k$, see Lemma~\ref{lem_distribution_on_Pk}.
In particular, the leaves of $G$ are algebraic and $\overline{\Gamma_c}$ is a union of one parameter family of algebraic curves,
  and thus $k=2$, which completes the proof of the claim of the proposition.
\end{prf}

In conclusion:

\begin{cor}\label{cor_family_of_planes_covering_B}
In Setting~\ref{setting_divisor_of_non-standard_lines}, 
the divisor $B$ is dominated by a family of linear subspaces of dimension~$2$ as defined in \S\ref{sect_lines_and_linear_subspaces}.
Let $\ccP$ be this family and $\ccU_{\ccP}$ be the universal family:
\[
  \xymatrix{ & \ccU_{\ccP} \ar[dr]^{\phi_{\ccP}} \ar[dl]_{\pi_{\ccP}}&\\ \ccP\ar@/_/@{.>}[ur]_{\upsilon}&& B.}
\]
The fibres of $\pi_{\ccP}$ are $\PP^2$ and the image of a general fibre $\phi_{\ccP}(\PP^2) \subset B$ is the surface $\overline{\Gamma_c}$ for some line $c \in \ccB$.
The restriction of $\phi_{\ccP}$ to $\PP^2 \to \overline{\Gamma_c}$ is the normalisation map.
The rational section $\upsilon \colon \ccP \dashrightarrow \ccU_{\ccP}$ is assigning to a general plane $\PP^2$ (the normalisation of the surface $\overline{\Gamma_c}$) its distinguished point $y$ 
   constructed in the proof of Proposition~\ref{prop_Fall}.
\end{cor}  
\noprf

\subsection{The final jump}
We claim the locus $B$ swept by $\ccB$ as above is covered by linear subspaces of dimension $n$.
Lemma~\ref{lem_dim_Y_eq_n} below implies the claim and contains more details about these subspaces.

Let $Y \subset B$ be the closure of the image of $\phi_{\ccP} \circ \upsilon  \colon \ccP \dashrightarrow \ccU_{\ccP} \to B$,
  that is the set of all distinguished points as in Corollary~\ref{cor_family_of_planes_covering_B}.
This locus of distinguished points is very important in our considerations below and it will be used to ``bundle'' linear subspaces of dimension $2$ into a linear subspace of dimension $n$.
Eventually we claim that $\dim Y =n$, but first we must observe that $Y \ne B$.  
  
\begin{lemma}\label{lem_Y_ne_B}
  With $X$, $\ccB$, $B$ as above and throughout this section,
  the locus $Y$ of distinguished points is not equal to $B$.
\end{lemma}
\begin{prf}
  Let $y \in B$ be the general point and suppose by contradiction that $y \in Y$,
    that is, there exists $\PP^2$ in $\ccP$ with the distinguished point mapped to $y$.
  Since $y$ is general, $G$ is defined at $y$.
  The images of lines in $\PP^2$ through the distinguished point form a one dimensional family of contact lines tangent to $G$.
  In particular, they share the tangent direction at $y$. This is impossible by Lemma~\ref{lem_bend_and_break_tangent_directions}.
\end{prf}

For $y \in Y$, let $\ccP^y$ be the closure of the preimage $(\phi_{\ccP} \circ \upsilon)^{-1}(y)$,
  that is, essentially, the set of the planes $\PP^2$ with $y$ as the distinguished point.
The locus $P^y \subset B$ of $\ccP^y$ is the union of $\overline{\Gamma_c}$ corresponding to the points in $\ccP^y$. 
This is our candidate for the linear space of dimension $n$.

\begin{lemma}\label{lem_dim_Y_eq_n}
  We work in Setting~\ref{setting_divisor_of_non-standard_lines}, and we let $Y \subset B$ be the locus of distinguished points as above,
   and for $y \in Y$, the sets $\ccP^y \subset \ccP$, $P^y \subset B$ are as defined in the paragraph preceding this lemma. 
  Then $\dim Y = n$ and for a general $y \in Y$,
    the locus $P^y$ is a component of $B_y$ (the locus swept by lines from $\ccB$ passing through $y$), whose normalisation is a $\PP^n$.
\end{lemma}
\begin{prf}
  We need to count the dimensions and relative dimensions of the spaces appearing in our construction.
  Firstly, a general $\overline{\Gamma_c}$ in $\ccP$ is uniquely determined by a general line $c \in \ccB$.
  Thus we have a dominant rational map $\ccB \dashrightarrow \ccP$.
  The fibres are two dimensional: two lines $c,c' \in \ccB$
     determine the same $\overline{\Gamma_c} = \overline{\Gamma_{c'}}$ if and only if the following three conditions are satisfied:
  \begin{itemize}
   \item both $C$ and $C'$ intersect the locus where $G$ is defined;
   \item both $C$ and $C'$ are transversal to $G$ at their general points;
   \item $C' \subset \overline{\Gamma_c}$
         (assuming the above two conditions, this is equivalent to $C \subset \overline{\Gamma_{c'}}$).
  \end{itemize}
  The first two conditions are open in $\ccB$.
  The final one says that $C'$ is an image of one of the lines in the normalisation $\PP^2 \to \overline{\Gamma_c}$, and there is a two dimensional family of such lines.
  Thus it follows that:
  \begin{align}
     \dim \ccP & = \dim \ccB - 2 = 3n-4, \text{ and} \nonumber\\
     \dim \ccP^y & = 3n-4 - \dim Y. \label{equ_equality_dim_ccP^y}
  \end{align}
  A priori, $\ccP^y$ could be reducible. In such a case Equation \eqref{equ_equality_dim_ccP^y}
    is about pure dimension: every irreducible component of $\ccP^y$ has dimension $3n-4 - \dim Y$.

  Let $\ccS \subset \ccB$ be the closure of the set of lines tangent to $G$.
  Note that $\dim \ccS = 2n-1$ since there is a unique line in $\ccS$ through each general point of $B$.
  Also for a general line $c \in \ccS$ the intersection $C \cap Y$ is non-empty and finite.
  To prove this claim, the general point $s \in C$ is a general point of $B$ and belongs to a general $\PP^2 \in \ccP$.
  Thus $C$ is the image of the line in the plane $\PP^2$, that connects $s$ and the distinguished point of $\PP^2$.
  This shows that the intersection $C \cap Y$ is non-empty.
  Also $C$ is not contained in $Y$, since otherwise $Y = B$ contrary to Lemma~\ref{lem_Y_ne_B}.
  Thus $C \cap Y$ is finite.

  Let $y \in Y$ be a general point and suppose that $\ccS_y \subset \ccS$ is the set of lines in $\ccS$ containing $y$.
  Then
  \begin{align}
     \dim \ccS_y & = \dim \ccS - \dim Y = 2n-1 -\dim Y \text{ and} \nonumber \\
     \dim \ccS_y &\le  \dim \ccB_y = n - 1, \text{ so that} \nonumber \\
     \dim Y &\ge n \label{equ_inequality_dim_Y_ge_n}.
  \end{align}
  Furthermore, consider the locus $S_y \subset B_y$ swept by these lines.
  Note that $P^y \subset S_y$, thus:
  \begin{align}
     \dim P^y & \le \dim S_y \le \dim \ccS_y +1 = 2n -\dim Y. \label{equ_inequality_dim_P^y}
  \end{align}

  Let $\ccU_{\ccP^y}$ be the restriction of $\ccU_{\ccP}$ to $\ccP^y$, so that the image is
    $P^y = \phi_{\ccP} (\ccU_{\ccP^y}) \subset B$.
  We also consider the fibre product
  \begin{align}
    \ccU_{\ccP^y}^2 &:= \ccU_{\ccP^y} \times_{\ccP^y} \ccU_{\ccP^y}, \nonumber\\
    \text{so that } \dim \ccU_{\ccP^y}^2 & = \dim \ccP^y + 4 \stackrel{\eqref{equ_equality_dim_ccP^y}}{=} 3n - \dim Y,
     \label{equ_equality_dim_ccU_ccP^y^2}
  \end{align}
  and its map to $P^y \times P^y$.
  Less formally, $\ccU_{\ccP^y}^2$ is the set of triples $(\PP^2, \tilde{u}, \tilde{v})$,
     with $\tilde{u}, \tilde{v} \in \PP^2$ and the triple is mapped to two points $u,v$ in $\overline{\Gamma}_c$
     which is the surface normalised by the plane $\PP^2$.
  The two points in $\overline{\Gamma_c}$ are the images of $\tilde{u}$ and $\tilde{v}$ under the normalisation map.
  We claim that the map $\ccU_{\ccP^y}^2 \to P^y \times P^y$ is generically finite onto its image.
  More precisely, the map is generically finite onto the image of each irreducible component of $\ccU_{\ccP^y}^2$.

  To prove the claim, suppose that there is a curve $Z \subset \ccU_{\ccP^y}^2$
     contracted to a single point $(u, v) \in P^y \times P^y$.
  Suppose moreover, that $Z$ contains a general point $z_0$ of $\ccU_{\ccP^y}^2$
     (more precisely, $z_0$ is a general point of any of the components).
  The generality conditions on $(\PP^2_{z_0}, \tilde{u}_{z_0}, \tilde{v}_{z_0}) \in Z$ that we need are:
  \begin{itemize}
     \item If $\tilde{y}_{z_0} \in \PP^2_{z_0}$ is the distinguished point,
           then $\tilde{u}_{z_0}, \tilde{v}_{z_0}, \tilde{y}_{z_0}$ are not on a line in $\PP^2_{z_0}$.
     \item $G$ is defined at $u_{z_0}$.
  \end{itemize}
  Then $Z$ determines a curve in $\ccP^y$, such that each $\PP^2_z$
     on this curve contains all three of the points $y, u, v$.
  In particular, we can take the family of lines connecting $u$ and $v$ on each of the planes $\PP^2$.
  By Lemma~\ref{lem_bend_and_break_1_point}, the family of lines must be constant.
  This is a contradiction, since each plane $\PP^2$ is uniquely determined by the line
     (because $G$ is defined at $u$, so in particular, it is defined at a general point of that line).

  Thus $\ccU_{\ccP^y}^2 \to P^y \times P^y$ is generically finite onto its image,
     and
  \begin{align}
     \dim \ccU_{\ccP^y}^2 &\le 2 \dim P^y. \label{equ_inequality_dim_ccU_ccP^y^2}
  \end{align}
  We summarise our dimension counts:
  \begin{align}
     3n -\dim Y& \stackrel{\eqref{equ_equality_dim_ccU_ccP^y^2}}{=} \dim \ccU_{\ccP_y}^2
                 \stackrel{\eqref{equ_inequality_dim_ccU_ccP^y^2}}{\le} 2 \dim P^y
                 \stackrel{\eqref{equ_inequality_dim_P^y}}{\le} 2 (2n -\dim Y) \label{equ_long_inequality_for_dim_Y}\\
     \text{thus } \dim Y & \le n \text{ and combining with Inequality \eqref{equ_inequality_dim_Y_ge_n}:} \nonumber\\
      \dim Y & =  n.  \label{equ_equality_dim_Y_eq_n}
  \end{align}
  Thus we obtain the first claim of the lemma. Moreover, rewriting \eqref{equ_long_inequality_for_dim_Y}:
  \[
     2n = \dim \ccU_{\ccP_y}^2
        \stackrel{\eqref{equ_inequality_dim_ccU_ccP^y^2}}{\le} 2 \dim P^y
        \stackrel{\eqref{equ_inequality_dim_P^y}}{\le} 2n  \\
  \]
  we obtain an equality in \eqref{equ_inequality_dim_P^y} and in \eqref{equ_inequality_dim_ccU_ccP^y^2}:
  \begin{align*}
     \dim P^y &= n\\
     \dim \ccU_{\ccP_y}^2 & =  2 \dim P^y.
  \end{align*}
  Since the map $\ccU_{\ccP^y}^2 \to P^y \times P^y$ is generically finite onto its image,
     the dimension count proves that the map is dominant.
  Equivalently, for two general points in $P^y$, there exists a $\PP^2$ in $\ccP^y$,
     whose image in $P^y$ contains both points.
  In particular, there exists a line connecting the two points.
  Thus the normalisation of $P^y$ is $\PP^n$ by Theorem \ref{thm_many_lines}, and the lemma is proved.
\end{prf}

The lemma completes the proof of the claim of the final jump.
Now we can conclude our article with the proof of Theorem~\ref{thm_dimension_of_the_space_of_special_lines} about the structure of the locus $B$.
\begin{proof}[Proof of Theorem~\ref{thm_dimension_of_the_space_of_special_lines}]
  With $\ccB$ and $B$ as above, we have shown in Lemma~\ref{lem_dim_Y_eq_n},
     that $B$ is a divisor covered by linear subspaces of dimension $n$.
  We claim that there is a unique such linear subspace through a general point of $B$.
  To see this we will construct a distribution $G' \subset T B$ such that each linear subspace is a leaf of $G'$.
    
  Consider the lines tangent to $G$. 
  They form a family of lines of dimension $2n-1$, which cover the divisor $B$.
  This is the same family as was denoted by $\ccS$ in the proof of Lemma~\ref{lem_dim_Y_eq_n}.
  Pick a general such line $C$.
  As in Section~\ref{sect_splitting_types}, 
     consider the subbundle $(TX|_C)^{+} = \bigoplus_{i=1}^n \ccO(a_i)$, which is the sum of positive line bundle summands,
     that is, $a_i >0$.
  Then this bundle has a constant (independent of $C$) rank $n$.
  Since there is a unique such line through a general point of $B$, the subbundles $(TX|_C)^{+}$ glue together to a rank $n$ distribution $G'$ in $TX|_B$.
  Each linear space is swept out by the deformations of $C$ with one point fixed. 
  Thus the tangent space to the linear space at its general point is equal to the fibre of $G'$.

  In particular, there is a unique linear space through a general point of $B$. 
  Consider the (normalised) family of linear spaces $\pi\colon\ccU_{\ccR} \to \ccR$ of dimension $n$, that dominates $B$.
  By Proposition~\ref{prop_linear_spaces_and_Chow_variety_for_projective} the base of the family might be chosen to be projective,
     the sheaf $\ccE := \pi_*(\xi^* L)$ is a vector bundle of rank $n+1$, and $\ccU_{\ccR} \simeq \PP(\ccE^*)$ with $\ccO_{\PP(\ccE^{*})}(1) = \xi^*L$.
  Moreover, we may assume that the morphism to Chow variety as in Proposition~\ref{prop_linear_spaces_and_Chow_variety}\ref{item_families_of_linear_spaces_determine_morphism_to_Chow}
     is a normalisation of its image, in particular finite.
  The above uniqueness argument shows that the evaluation map $\xi$ is birational.
  Moreover, we claim below that $\xi$ is finite.
  The claim (proved in the next paragraph) implies, that $\xi$ is the normalisation map of $B$, thus completes the proof of 
     items \ref{item_normalisation_of_B_is_a_Pn_bundle} and \ref{item_normalisation_of_B_normalises_linear_spaces} of the theorem.
  
  Let $\ccU_{\ccR}\xrightarrow{\beta} \ccU' \xrightarrow{\alpha} B$ be the Stein factorisation of $\xi$ with $\beta$ being a projective morphism with connected fibres, $\alpha$ a finite morphism, and $\ccU'$ normal.
  We claim that $\beta$ is an isomorphism.
  To see that, suppose by contradiction that $\xi$ contracts an irreducible closed positive dimensional subvariety $Z \subset \ccU_{\ccR}$ 
     to a point $x \in X$.
  If $Z$ is contained in some fibre $\PP^n$, then $\xi^*L|_{Z}$ is ample by our assumption on $\xi^*L$
     and trivial by our choice of $Z$, a contradiction.
  Thus $Z$ maps onto a closed positive dimensional and irreducible $\tilde Z \subset \ccR$.
  Let $\ccU_{\tilde{Z}} \subset \ccU_{\ccR}$ be the restricted family.
  The image $\xi(\ccU_{\tilde{Z}})$ is an irreducible subset of $X$, whose points are connected to $x$ by lines.
  Therefore $\xi(\ccU_{\tilde{Z}}) \subset H_x$ and by dimension count
     $\xi(\ccU_{\tilde{Z}}) = \xi(\ccU_{\tilde{z}})$, for any $\tilde{z} \in \tilde{Z}$.
  In particular, $\tilde{Z}$ is contracted under the morphism from $\ccR$ to the Chow variety.
  This is impossible, since the morphism is finite (it is a normalisation of a subset of the Chow variety), 
     and $\dim \tilde{Z} > 0$.
  
  It remains to show \ref{item_B_nonnormal}. 
  Recall from above, that $\dim B =2n$ and $\ccU_{\ccR}$ is the normalisation of $B$ by \ref{item_normalisation_of_B_is_a_Pn_bundle}.
  Moreover, the second Betti number $b_2(\ccU_{\ccR}) \ge 2$ since $\ccU_{\ccR}$ is a projective space bundle over a projective base,
     while $b_2(B)=b_2(X)=1$ by Lefschetz hyperplane section theorem
     \cite[Thm~2.3.1]{beltrametti_sommese_adjunction_theory_of_projective_varieties}. 
  Note that the divisor $B$ on $X$ is ample, since it is effective and $\Pic X \simeq \ZZ$.
  In particular, $\ccU_{\ccR}$ is not isomorphic to $B$ and thus $B$ is not normal as claimed.
\end{proof}

\bibliography{special_contact_lines}
\bibliographystyle{alpha_four}

\end{document}